\documentclass{conm-p-l}
\usepackage{amssymb}

\usepackage{amsmath}
\usepackage{liemacs-red}

\newcommand\cQ{\mathcal{Q}} 
\newcommand\cR{\mathcal{R}} 
\newcommand\hg{{\rm ht}} 

\newcommand{\oalpha}{{\oline{\alpha}}}

\copyrightinfo{2008}{American Mathematical Society}

\newtheorem{theorem}{Theorem}[section]
\newtheorem{lemma}[theorem]{Lemma}
\newtheorem{prop}[theorem]{Proposition}
\newtheorem{corollary}[theorem]{Corollary} 

\theoremstyle{definition}
\newtheorem{definition}[theorem]{Definition}
\newtheorem{example}[theorem]{Example}

\theoremstyle{remark}
\newtheorem{remark}[theorem]{Remark}
\newtheorem{prob}[theorem]{Problem} 

\numberwithin{equation}{section}

\begin{document}

\title[Unitary highest weight modules]
{Unitary highest weight modules of locally affine Lie algebras}


\author[Neeb]{Karl-Hermann Neeb} 
\address{Fachbereich Mathematik, TU Darmstadt, Schlossgartenstrasse 7, 
64289-Darmstadt, Germany} 
\curraddr{}
\email{neeb@mathematik.tu-darmstadt.de}
\thanks{}

\subjclass[2000]{Primary 17B70, Secondary 17B10}

\date{10.9.2008}

\begin{abstract} Locally affine Lie algebras are 
generalizations of affine Kac--Moody algebras with Cartan 
subalgebras of infinite rank whose root system is locally affine. 
In this note we study a class of representations of locally 
affine algebras generalizing integrable highest weight 
modules. In particular, we construct such an integrable representation 
for each integral weight not vanishing on the center and 
show that, over the complex numbers, we thus obtain 
unitary representations w.r.t.\ a unitary real form. 

We also use Yoshii's recent classification of locally affine root 
systems to derive a classification of so-called minimal 
locally affine Lie algebras and give realizations as twisted loop algebras. 
\end{abstract} 

\maketitle

\section*{Introduction}  \label{sec:0} 

It is an important feature of the integrable highest weight modules 
$L(\lambda)$ of a Kac--Moody Lie algebra $\g(A)$, that the set 
$\cP_\lambda$ of weights can be described as 
\begin{equation}
  \label{eq:1}
\cP_\lambda = \conv(\cW\lambda) \cap (\lambda + \cQ),
\end{equation}
where $\cQ = \Spann_\Z \Delta$ 
is the root lattice. This description implies in 
particular that $\conv(\cP_\lambda) = \conv(\cW\lambda)$ and that 
each weight $\lambda$ is an extreme point of this set. 
Over the field $\C$ of complex numbers, the Lie algebra $\g(A)$ 
has a natural antilinear involution defining a so-called 
{\it unitary real form} $\fk(A)$, and if $A$ is symmetrizable, each 
integrable highest weight module $L(\lambda)$ is unitary in the sense 
that it carries a positive definite $\fk(A)$-invariant hermitian form 
(\cite{Ka90}). If $A$ is of affine (or finite) type, 
then one even has natural realizations of these unitary modules 
in spaces of holomorphic sections of holomorphic line bundles 
over certain coadjoint orbits with K\"ahler structures 
(cf.\ \cite{PS86}, \cite{Ne01}, where this is discussed for loop groups). 

On the other hand, split locally finite Lie algebras $(\g,\fh)$ 
(i.e., $\g$ has a root decomposition with respect to $\fh$)  
also have a natural class of representations fitting into this scheme. 
In this context the situation becomes more tricky because 
one cannot choose a positive system, resp., a set of simple roots, 
a priori, but for each integral 
weight $\lambda$ there exists a positive system 
for which $\lambda$ is dominant. As we have seen in \cite{Ne98}, 
this also leads to integrable highest weight modules $L(\lambda)$ 
which are uniquely determined by their weight set 
$\cP_\lambda$, given by \eqref{eq:1}. In the complex case, they are unitary 
with respect to a unitary real form of the Lie algebra $\g$.  
In particular, these modules do not depend on the positive system 
with respect to which they are dominant, and their equivalence 
classes are parameterized by the set $\cP/\cW$ of Weyl group orbits 
in the set $\cP$ of integral weights (which in finite dimensions   
is usually identified with the set of dominant integral weights  
with respect to 
some positive system $\Delta^+$). 
Under suitable boundedness conditions on the weight, the 
realization of these modules in holomorphic line bundles is discussed 
in the context of Banach--Lie groups in \cite{Ne04}. 

Comparing these two classes of representations which exhibit 
a rich geometric structure, we would like to understand if there are 
larger classes of Lie algebras with similar types of representations, 
or even if there is 
a common roof for the Kac--Moody cases and the locally finite 
situation. In this note we address this question for the class of 
locally affine Lie algebras (LALAs), a subclass of the 
recently introduced locally extend affine Lie algebras 
(LEALAs) (\cite{MY06}), which in turn are infinite rank variants 
of extended affine Lie algebras (\cite{AA-P97}). 

From a geometric point of view, the choice of a positive system with respect to which 
an integral weight $\lambda$ is dominant is a redundant 
feature of the theory. It is much more natural to work 
with the stabilizer $\fp_\lambda$ of the highest weight ray in 
$L(\lambda)$. 
A key observation that makes our approach work is that for an affine 
Kac--Moody algebra, the property of an integral weight $\lambda$ 
to be an extremal weight of an integrable highest weight module 
is equivalent to $\lambda$ not vanishing on the center, a property 
that no longer refers to the choice of a positive system, so that 
it also makes sense for locally affine Lie algebras. 

In the first section we introduce the framework of split quadratic 
Lie algebras $(\g,\fh,\kappa)$, where $\kappa$ is an invariant 
symmetric bilinear form and $\fh$ is a splitting Cartan subalgebra. 
In this context one has a natural concept of an integrable root 
generalizing the notion of 
a real root of a Kac--Moody Lie algebra, resp., an anisotropic root 
of an EALA. In this framework we define LEALAs and LALAs following 
\cite{MY06} and \cite{YY08}. Since notions such as 
the Weyl group, coroots and integral 
weights make sense in general, it becomes an interesting issue 
to understand the interactions of the axiomatic framework 
and geometric, resp., representation theoretic features of these  
Lie algebras. The main new observation in Section~\ref{sec:1} 
is that 
any integral weight $\lambda$ defines a natural $\Q$-grading 
$\g = \bigoplus_{q \in \Q} \g^q(\lambda)$ and that this becomes a 
$\Z$-grading if 
\begin{itemize}
\item[\rm(2-Aff)] $\beta(\check \alpha)\alpha(\check\beta) \leq 4$ 
for $\alpha,\beta \in \Delta_i$,  
\end{itemize}
where $\Delta_i$ denotes the set of integrable roots 
(Definition~\ref{def:basic}). 
This means that the  subalgebras $\g(\alpha,\beta)$ 
generated by two integrable root-$\fsl_2$-algebras 
are either finite dimensional or affine Kac--Moody. 

In Section~\ref{sec:2} we then turn to LEALAs. Since all these Lie algebras 
satisfy (2-Aff), each integral weight $\lambda$ defines a $\Z$-grading 
of $\g$. For affine Kac--Moody Lie algebras, the $\Z$-gradings corresponding 
to integrable highest weight modules $L(\lambda,\Delta^+)$ 
have the property that all roots of the $0$-component $\g^0(\lambda)$ 
of the grading are integrable. 
We therefore focus on integral weights with this 
property. The first main result of Section~\ref{sec:2} is that 
the existence of an integral weight with this property 
implies that $\g$ is a locally affine Lie algebra (Theorem~\ref{thm:2.8}). 
This is a representation theoretic 
characterization of LALAs among LEALAs. From this point on 
we restrict our considerations to locally affine Lie algebras.  

In Section~\ref{sec:3} we turn to the structure of locally 
affine Lie algebras. After analyzing how they can be exhausted by 
affine Kac--Moody algebras in a controlled fashion, we 
show that 
isomorphisms of locally affine root systems ``extend'' to isomorphisms 
of the cores of the corresponding Lie algebras and even to all of 
$\g$, provided $\g$ is {\it minimal} 
(cf.\ Definition~\ref{def:min}).\begin{footnote}
  {After this paper was written, we learned about the work in progress 
\cite{MY08} of Morita and Yoshii, where they also show that 
the cores or locally affine Lie algebras with isomorphic root systems 
are isomorphic. However, their approach is somewhat different.} 
\end{footnote}
From that it also follows that for a minimal locally affine 
Lie algebra two Cartan subalgebras 
with isomorphic root systems are conjugate under $\Aut(\g)$ 
(cf.\ \cite{Sa08} for related results)  
and that, over $\C$, minimal affine Lie algebras have a unitary 
real form. 

Section~\ref{sec:4} is completely devoted to representation theoretic 
issues. Here we construct for each integral weight $\lambda$ of a 
locally extended affine Lie algebra $\g$ a simple module 
$L(\lambda) := L(\lambda, \fp_\lambda)$ by induction of a generalized 
parabolic subalgebra $\fp_\lambda = \sum_{q \geq 0} \g^q(\lambda)$. 
Our main representation theoretic result 
is that these modules $L(\lambda)$ 
are integrable and that, over $\C$, they are unitary with respect to 
the unitary real form. Moreover, $\cP_\lambda$ is given by 
\eqref{eq:1} and $L(\lambda) \cong L(\mu)$ if and only if 
$\mu$ is conjugate to $\lambda$ under the Weyl group $\cW$. 

In a first appendix we recall Yoshii's classification of the locally affine 
root systems of infinite rank, derive some more information that is 
needed in Section~\ref{sec:3} and describe a realization of the 
corresponding minimal locally affine Lie algebras in terms 
of twisted loop algebras. In a second appendix we show that 
Yoshii's seven types of infinite rank locally affine root systems 
lead to four isomorphy classes of 
minimal locally affine Lie algebras. 

In the following $\K$ denotes a field of characteristic zero, 
if not explicitly said otherwise.

{\bf Acknowledgement:} We thank Y.~Yoshii for illuminating discussions 
during the preparation of this manuscript 
and for making the manuscripts \cite{MY08} and \cite{YY08} available.

\section{Split quadratic Lie algebras} \label{sec:1} 

\begin{definition} \label{def:basic} (a) We call an abelian subalgebra 
$\fh$ of the Lie algebra $\g$ a {\it splitting Cartan subalgebra} if 
$\fh$ is maximal abelian and $\ad \fh$ is simultaneously
diagonalizable. Then the pair $(\g,\fh)$, resp., $\g$, 
is called a {\it split Lie algebra} and 
we have a {\it root decomposition} 
$$ \g = \fh + \sum_{\alpha \in\Delta} \g_\alpha, $$
where $\g_\alpha = \{ x \in \g \: (\forall h \in \fh) [h,x]= \alpha(h)
x\}$ and 
$$\Delta := \Delta(\g,\fh) := \{ \alpha \in \fh^* \setminus \{0\} \: \g_\alpha
\not= \{0\}\}$$ is the corresponding {\it root system}. 
Note that $\g_0 = \fh$ because $\fh$ is maximal abelian. 

(b) We call a root $\alpha \in \Delta$ {\it integrable} 
if there exist $x_{\pm \alpha} \in \g_{\pm\alpha}$ with \break 
$\alpha([x_\alpha, x_{-\alpha}]) \not=0$ such that the operators 
$\ad x_{\pm \alpha}$ on $\g$ are locally nilpotent. 
Then $\dim \g_{\pm \alpha} = 1$ and we write $\check \alpha \in 
[\g_\alpha,\g_{-\alpha}]$ for the unique element satisfying 
$\alpha(\check \alpha) = 2$,  called the corresponding 
{\it coroot}. We also write 
$$ \g(\alpha) := \K \check \alpha + \g_\alpha + \g_{-\alpha} 
\cong \fsl_2(\K) $$
for the associated three dimensional subalgebra 
(cf.\ \cite[Prop.~I.6]{Ne00b}). 
The set of integrable roots is denoted $\Delta_i$. 
For a subset $S \subeq \Delta_i$, we write $\g(S)$ for the 
subalgebra generated by the $\g(\alpha)$, $\alpha \in S$. 
The subalgebra $\g_c := \g(\Delta_i)$ is called the {\it core of $\g$}. 
In the following, we shall frequently assume that $\g$ is 
{\it coral}, i.e., $\g = \fh + \g_c$. Then $\g_c$ is a perfect ideal of $\g$, 
hence coincides with the commutator algebra $[\g,\g]$. 

(c) A linear functional $\lambda\in \fh^*$ is called an {\it integral weight} 
if $\lambda(\check \alpha) \in \Z$ for each $\alpha \in \Delta_i$. 

(d) The Weyl group $\cW$ of $\g$ is the subgroup of 
$\GL(\fh)$ generated by the reflections 
$r_\alpha(h) := h - \alpha(h) \check \alpha.$
Then each $r_\alpha$ extends to an automorphism of $\g$ preserving 
$\fh$ and acting on $\fh^*$ by the adjoint linear map 
$r_\alpha(\beta) := \beta - \beta(\check \alpha) \alpha.$
The action of the Weyl group preserves the root system $\Delta$ 
and the subset $\Delta_i$ of integrable roots 
(cf.\ \cite[Def.~I.8, Lemma I.11]{Ne00b}). 

(e) We call two integrable roots $\alpha$ and $\beta$ {\it connected} if 
there exists a chain 
$\alpha_0= \alpha, \alpha_1, \ldots, \alpha_n = \beta$ 
in $\Delta_i$ with $\alpha_{i+1}(\check \alpha_i)\not=0$ 
for $i = 0,\ldots, n-1$. Then we define 
$\dist(\alpha,\beta)$ as the minimal length $n$ of such a chain. 
Connectedness defines an equivalence relation on 
$\Delta_i$ (cf.\ Proposition~\ref{prop:introot} below) 
and on each connected component of $\Delta_i$ 
(the corresponding equivalence 
classes), $\dist$ defines a metric. 
\end{definition} 

\begin{definition} A {\it quadratic Lie algebra} is a pair $(\g,\kappa)$, 
consisting of a Lie algebra $\g$ and a non-degenerate invariant symmetric 
bilinear form $\kappa$ on $\g$. 
If $(\g,\fh,\kappa)$ is a split quadratic Lie algebra, 
then the root spaces satisfy  
$\kappa(\g_\alpha, \g_\beta) = \{0\}$ for $\alpha + \beta\not=0$, and 
in particular, $\kappa\res_{\fh \times \fh}$ is non-degenerate. 
We thus obtain 
an injective map $\flat \: \fh \to \fh^*, h \mapsto 
h^\flat, h^\flat(x) := \kappa(x,h)$. For $\alpha\in 
\fh^\flat := \flat(\fh)$ 
we put $\alpha^\sharp := \flat^{-1}(\alpha)$ and define a symmetric 
bilinear map on $\fh^\flat$ by 
$(\alpha,\beta) := \kappa(\alpha^\sharp, \beta^\sharp)$. 
Note that (1.2) in Remark~\ref{rem:1.4} below implies 
in particular that $\Delta \subeq \fh^\flat$, so that 
$(\alpha,\beta)$ is defined for $\alpha,\beta \in \Delta$. 

A split quadratic Lie algebra $(\g,\fh,\kappa)$ 
is called a {\it locally extended affine Lie algebra} (LEALA) if  the 
following two conditions are satisfied 
\begin{itemize}
\item[\rm(Irr)] $\Delta_i$ is connected.
  \begin{footnote}
    {This is equivalent to $\Delta_i$ being {\it irreducible} 
in the sense that it cannot be decomposed into two proper mutually  
orthogonal subsets. 
}  \end{footnote}

\item[\rm(Int)] All non-isotropic roots are integrable, 
i.e., $\{ \alpha \in \Delta \: (\alpha,\alpha) \not=0\} \subeq \Delta_i$. 
\end{itemize}
\end{definition}  

The following proposition combines Cor.~III.8 and Prop.~III.9 in \cite{Ne00b}. 
\begin{prop} \label{prop:introot} For $\alpha, \beta \in \Delta_i$, 
the following assertions hold: 
\begin{itemize}
\item[\rm(i)] $\alpha(\check \beta) \in \Z$, 
$\alpha(\check \beta) \beta(\check \alpha) \geq 0$, and 
$\alpha(\check \beta) =0$ implies $\beta(\check \alpha)=0$. 
\item[\rm(ii)] If $\kappa$ is an invariant non-degenerate symmetric 
bilinear form on $\g$, then $\kappa(\check\alpha, \check \alpha) \not=0$,  
and if $\beta(\check \alpha) \not=0$, then 
$$\kappa(\check\beta,\check \beta)/\kappa(\check \alpha, \check \alpha) 
= \alpha(\check \beta)/\beta(\check \alpha) \in \Q^\times_+ := \{ q \in 
\bQ \: q > 0\}. $$
\end{itemize}
\end{prop}

\begin{remark} \label{rem:1.4} 
Let $(\g,\fh, \kappa)$ be a split quadratic Lie algebra. 

(a) For $h \in \fh$ and $x_\pm \in \g_{\pm\alpha}$ we have  
\begin{equation}
  \label{eq:brela}
\alpha(h)\kappa(x_\alpha, x_{-\alpha})
= \kappa([h,x_\alpha], x_{-\alpha})
= \kappa(h,[x_\alpha, x_{-\alpha}]), 
\end{equation}
so that the non-degeneracy of $\kappa$ on $\fh$ leads to $\alpha\in \fh^\flat$ 
and 
\begin{equation}
  \label{eq:brack-rel}
[x_\alpha, x_{-\alpha}] = \kappa(x_\alpha, x_{-\alpha}) \alpha^\sharp   
\end{equation}
(cf.\ \cite{MY08}).\begin{footnote}
{This argument shows that $(\g,\fh,\kappa)$ 
is an admissible triple in the sense of \cite{MY06}.}
  \end{footnote}
If $\alpha$ is integrable and $\check \alpha = [x_\alpha, x_{-\alpha}],$ 
then \eqref{eq:brela} 
and $\alpha(\check \alpha) =2$ imply 
$\kappa(\check \alpha, \check \alpha) = 2\kappa(x_\alpha, x_{-\alpha})$, 
which leads for $\beta \in \fh^\flat$ to 
\begin{equation}
  \label{eq:intflat}
\alpha^\sharp = \frac{2 \check \alpha}{\kappa(\check \alpha, 
\check \alpha)}, \quad 
(\alpha, \alpha) = \frac{4}{\kappa(\check \alpha, \check \alpha)} \quad 
\mbox{ and } \quad 
(\beta, \alpha) = \frac{2\beta(\check \alpha)}{\kappa(\check\alpha, 
\check \alpha)}.
\end{equation}

(b) If, in addition, $\Delta_i$ is connected, then Proposition~\ref{prop:introot}(ii) 
implies that, after multiplication with a suitable element 
of $\K^\times$, $\kappa$ satisfies 
$\kappa(\check\alpha, \check \alpha) \in \Q^\times_+$ for each 
$\alpha \in \Delta_i$, and thus 
$(\alpha,\alpha) \in \Q^\times_+.$ 
In the following, we shall always assume this normalization whenever 
$\Delta_i$ is connected. 

(c) It is easy to see that for $S \subeq \Delta_i$ we have 
\begin{equation}
  \label{eq:core-root}
\fh \cap \g(S) = \Spann \check S \quad \mbox{ and in particular\ }\quad 
\fh \cap \g_c = \Spann\check \Delta_i 
\end{equation}
(cf.\ \cite[Lemma~I.10]{Ne00b}). If $\g = \fh + \g_c$, then we further have 
\begin{equation}
  \label{eq:rootrel}
\Delta \subeq \Spann_\Z \Delta_i 
\quad \mbox{ and } \quad (\forall \beta\in \Delta)(\exists \alpha \in \Delta_i) \ \beta - \alpha \in \Delta.
\end{equation}
In fact, for each $\beta  \in \Delta$, the root space 
$\g_\beta$ is spanned by brackets of 
root vectors of integrable roots. Hence there exists an 
$\alpha \in \Delta_i$ with $\beta - \alpha \in \Delta$. 
From \eqref{eq:core-root} we derive that 
\begin{equation}
  \label{eq:orthog}
\g_c^\bot = \fh \cap (\check\Delta_i)^\bot 
= \Delta_i^\bot = \Delta^\bot = \fz(\g). 
\end{equation}
\end{remark} 

\begin{lemma} \label{lem:semdef} 
Let $(\g,\kappa)$ be a split quadratic Lie algebra with 
$(\alpha,\alpha) \in \Q^\times_+$ for each 
integrable root $\alpha$. 
For $\alpha, \beta \in \Delta_i$, the form  $(\cdot,\cdot)$ 
is positive semidefinite on $\Spann_\Q\{\alpha,\beta\}$ if and 
only if 
\begin{equation}
  \label{eq:estimate}
\alpha(\check \beta)\beta(\check \alpha) 
\in\{0,1,2,3,4\}. 
\end{equation} 
\end{lemma}

\begin{proof} If $\alpha$ and $\beta$ are linearly dependent over $\Q$, 
then $\beta \in \{\pm \alpha\}$ (\cite[Prop.~I.6]{Ne00b}), which implies 
that 
$\beta(\check \alpha) \alpha(\check \beta) = 4.$

Now we assume that $\alpha$ and $\beta$ are linearly independent. 
If $\beta(\check \alpha) = \alpha(\check \beta)~=~0$, then 
$(\alpha,\beta) = 0$, so that $(\cdot,\cdot)$ is 
positive definite on $\Spann_\Q\{\alpha,\beta\}$. We may therefore assume that 
$(\alpha,\beta) \not=0$. 
Then the assertion 
is equivalent to the positive semidefiniteness of the rational Gram matrix 
$\pmat{(\alpha,\alpha) & (\alpha,\beta) \cr 
(\alpha,\beta) & (\beta,\beta)},$
which means that 
$(\alpha,\beta)^2 \leq (\alpha, \alpha) (\beta, \beta).$
In view of Remark~\ref{rem:1.4}(a), this 
is equivalent to 
$\alpha(\check \beta)\beta(\check \alpha) \leq 4.$
Since both factors on the left are integers of the same sign, 
their product is non-negative, and the assertion follows 
(Proposition~\ref{prop:introot}(i)). 
\end{proof}

\begin{lemma}\label{lem:neighbor} If $\alpha,\beta \in \Delta_i$ 
are connected, then the orbit $\cW\beta$ contains an element $\gamma$ with 
$\gamma(\check \alpha) \not=0$. 
\end{lemma}

\begin{proof} If $\dist(\alpha,\beta) \leq 1$, then $\beta(\check\alpha)\not=0$, 
and there is nothing to show. If $\dist(\alpha,\beta) > 1$, we show that 
there exists an element $\gamma \in \cW\beta$ with 
$\dist(\gamma,\alpha) < \dist(\alpha,\beta)$. Then the assertion follows 
by induction. 

Let $\alpha_0 = \alpha, \alpha_1, \ldots, \alpha_n= \beta$ be a minimal chain of integrable roots with 
$\alpha_i(\check \alpha_{i+1}) \not=0$ for $i = 0,\ldots, n-1$. 
We consider the root 
$$ \gamma := r_{\alpha_{n-1}}(\beta) = \beta - \beta(\check \alpha_{n-1}) 
\alpha_{n-1}  \in \Delta_i $$
(cf.\ Definition~\ref{def:basic}(d)). 
Since $n$ is minimal and $>1$, we have $\beta(\check \alpha_{n-2})=~0$ and 
therefore 
$\gamma(\check \alpha_{n-2}) 
= - \beta(\check \alpha_{n-1}) \alpha_{n-1}(\check\alpha_{n-2}) \not=0.$
This implies that $\dist(\alpha,\gamma) <~n$. 
\end{proof}

\begin{prop} \label{prop:quotient} 
Suppose that the split quadratic Lie algebra $(\g,\fh,\kappa)$ satisfies 
{\rm(2-Aff)}. If $\alpha$ and $\beta$ are connected with 
$(\beta,\beta)/(\alpha, \alpha) > 1$, then 
$(\beta,\beta)/(\alpha,\alpha) \in \{2,3,4\}$. 
If, in addition, $\Delta_i$ is connected, then 
$(\alpha,\alpha)$ has at most $3$ values for $\alpha \in \Delta_i$. 
\end{prop}

\begin{proof} In view of Remark~\ref{rem:1.4}(b), we may w.l.o.g.\ assume that 
  $(\alpha,\alpha) \in \Q^\times_+$, so that the 
connectedness of $\beta$ with $\alpha$ implies 
$(\beta,\beta) \in \Q^\times_+$. 
Assume that $(\beta,\beta) > (\alpha,\alpha)$. 
In view of Lemma~\ref{lem:neighbor}, 
there exists $\gamma \in \cW\beta$ with $\gamma(\check \alpha) \not=0$. 
Since the form $(\cdot,\cdot)$ on $\Spann_\Z \Delta$ is invariant 
under the action of the Weyl group (which is easily checked on the 
generators with Remark~\ref{rem:1.4}(a)), we have 
$(\gamma,\gamma) = (\beta,\beta) > (\alpha,\alpha).$
By combining Proposition~\ref{prop:introot}(ii) with Remark~\ref{rem:1.4}, 
we find that 
$$ \frac{(\beta,\beta)}{(\alpha,\alpha)} 
= \frac{(\gamma, \gamma)}{(\alpha,\alpha)} 
= \frac{\kappa(\check\alpha, \check \alpha)}{\kappa(\check\gamma,\check\gamma)}
= \frac{\gamma(\check \alpha)}{\alpha(\check \gamma)}. $$
Therefore $|\alpha(\check\gamma)| < |\gamma(\check \alpha)|$, so that the 
integrality of these numbers and (2-Aff) lead to 
$|\alpha(\check\gamma)| =1$. This proves that 
$\frac{(\beta,\beta)}{(\alpha,\alpha)} 
= |\gamma(\check \alpha)| \in \{2,3,4\}.$

Now we assume, in addition, that $\Delta_i$ is connected. 
We may w.l.o.g.\ assume that the minimal value of 
$(\alpha, \alpha)$ for $\alpha \in \Delta_i$ is $1$. 
Then the argument above  implies that 
$(\beta,\beta) \in \{1,2,3,4\}$ for $\beta \in \Delta_i$. 
If there exists a $\beta$ with $(\beta,\beta) = 3$, then 
$(\gamma,\gamma) \in \{2,4\}$ is ruled out for $\gamma \in \Delta_i$ by 
the preceding argument. Therefore the only possibility 
for the set of square lengths is $\{1,3\}$ or a subset of $\{1,2,4\}$. 
\end{proof}

The preceding proposition has an interesting consequence. 

\begin{theorem} \label{thm:q-grad} 
Let $(\g,\fh,\kappa)$ be a split quadratic Lie algebra 
with $\kappa(\check\alpha,\check\alpha)\in\Q$ for 
$\alpha\in \Delta_i$ and $\Delta \subeq \Spann_\Z \Delta_i.$
Then each integral weight $\lambda \in \fh^*$ defines a 
$\Q$-grading of $\g$ by 
$$ \g^q := \g^q(\lambda) 
:= \bigoplus_{\lambda(\alpha^\sharp)= q} \g_\alpha \quad \mbox{ 
for } \quad q \in \Q. $$
If, in addition, {\rm(2-Aff)} holds and $\Delta_i$ is connected, 
then the support of this 
grading lies in a cyclic subgroup of $\Q$, so that we actually 
obtain a $\Z$-grading. 
\end{theorem}

\begin{proof} Clearly, the map $\fh^\flat \to \K, \alpha \mapsto 
\alpha^\sharp(h)$, 
is additive, and for $\alpha \in \Delta_i$ we have 
$\lambda(\alpha^\sharp) = \frac{2 \lambda(\check \alpha)}
{\kappa(\check \alpha,\check \alpha)} \in \Q.$
Now the first assertion follows from $\Delta\subeq \Spann_\Z(\Delta_i)$. 

Let us now assume that (2-Aff) is satisfied and $\Delta_i$ is connected. 
We have to show that the set of rational numbers of the form 
$$ \lambda(\alpha^\sharp) 
= \frac{2\lambda(\check\alpha)}{\kappa(\check \alpha, \check \alpha)} 
= \frac{1}{2}(\alpha,\alpha)\lambda(\check\alpha) 
\in \frac{1}{2} (\alpha, \alpha) \Z $$ 
is contained in a cyclic subgroup of $(\Q,+)$. 
Since each finitely generated 
subgroup of $\Q$ is cyclic, this follows from Proposition~\ref{prop:quotient}. 
\end{proof}

In the following, we 
are mainly interested in those integral weights for which all 
roots $\alpha$ with $\lambda(\alpha^\sharp) = 0$ are integrable. 
The existence of such weights for LEALAs will be investigated in 
Section~\ref{sec:2} below.

The following proposition completely describes the meaning of 
(2-Aff) for Kac--Moody algebras. 
\begin{footnote}
  {We thank Pierre-Emmanuel Caprace for suggesting a different proof 
of this result, based on the geometry of the corresponding Coxeter group.}
\end{footnote}

\begin{prop} A Kac--Moody algebra $\g(A)$ 
satisfies {\rm(2-Aff)} if and only if it either is finite-dimensional 
or affine. 
\end{prop}

\begin{proof} If $\g(A)$ is affine or finite-dimensional, then 
the canonical bilinear form on $\Spann \Delta$ is positive semidefinite, 
so that (2-Aff) follows from Lemma~\ref{lem:semdef}. 

Suppose, conversely, that $\g(A)$ is infinite dimensional and 
satisfies (2-Aff). 
Let $\Pi = \{ \alpha_1, \ldots, \alpha_r\} 
\subeq \Delta_i$ be a system of simple roots 
and define the height of a root by 
$\hg(\sum_i n_i \alpha_i) := \sum_i n_i$. 
Let $\delta = \sum_{i = 1}^r n_i \alpha_i 
\in \Delta^+$ be a non-integrable root of minimal height 
(cf.\ \cite{Ka90}). 
Then there exists a simple root $\alpha \in \Pi$ with 
$\beta := \delta - \alpha \in \Delta$, and since the height of 
$\delta$ is minimal, $\beta$ is integrable.  
Next we use \cite[Prop.~III.14]{Ne00b} to see that the 
subalgebra $\g(\alpha,\beta)$ generated by 
$\g(\alpha)$ and $\g(\beta)$ is the commutator algebra of an 
affine Kac--Moody algebra of rank $2$. In particular, 
for each $n \in \N$, $\alpha + n\delta\in \Delta_i$ 
(\cite[Thm.~III.5]{Ne00b}). 
Since (2-Aff) implies in particular that $|\gamma(\check\eta)| \leq 4$ 
for $\gamma,\eta \in \Delta_i$, it follows that 
$\delta(\check \eta) = 0$ for each integrable root $\eta$. 

If some coefficient $n_{i}$ of $\delta$ vanishes, then 
the connectedness of $\Pi$ implies the existence of some 
$j$ with $n_j > 0$ and $\alpha_j(\check \alpha_{i}) < 0$. 
Since $\alpha_i(\check \alpha_{k}) \leq 0$ holds for each $i\not=k$, 
we thus arrive at the contradiction $\delta(\check \alpha_{i}) < 0$.  
Therefore the vector ${\bf n} = (n_1,\ldots, n_r)^\top \in \N^r$ 
has no zero entry and 
satisfies $A {\bf n}= 0$, so that Vinberg's classification of generalized 
Cartan matrices (\cite[Prop.~3.6.5]{MP95}) implies that $A$ is of affine 
type. 
\end{proof}

We record the following proposition because we shall use it later 
on to obtain a classification of minimal locally affine 
Lie algebras. 

\begin{definition} A subset $\Pi \subeq \Delta$ is called a {\it simple system} 
if $\alpha-\beta \not\in \Delta$ for $\alpha,\beta \in \Pi$.   
\end{definition}

\begin{prop}\label{prop:1.10} Let $(\g,\fh,\kappa)$ be a split 
quadratic Lie algebra 
and $\Pi = \{ \alpha_1, \ldots, \alpha_r\} \subeq \Delta_i$ 
be a linearly independent simple system. Then the following 
assertions hold: 
\begin{itemize} 
\item[\rm(i)] $\check \Pi$ is linearly independent in $\fh$. 
\item[\rm(ii)] The matrix 
$A_\Pi := (\alpha_i(\check \alpha_j))_{i,j \in \Pi}$ is a 
symmetrizable generalized Cartan matrix. 
\item[\rm(iii)] Let $k := \rk(A_\Pi)$ and $n := 2r-k$. 
Put $h_i := \check \alpha_i$ 
for $i =1,\ldots, r$ and choose elements 
$h_{r+1},\ldots, h_n$ such that $h_1,\ldots, h_n$ are linearly independent 
and their span $\fh_\Pi$ separates the points in $\Spann \Pi$. 
Then 
$$ \fh_\Pi + \g(\Pi) \cong \g(A_\Pi), $$
the Kac--Moody Lie algebra associated to $A_\Pi$. 
\item[\rm(iv)] If, in addition, $A_\Pi$ is of affine type, then 
$n = r + 1$ and we may choose any $h_n \in \fh$ on  which 
the isotropic roots of $\g(A_\Pi)$ do not vanish. 
\end{itemize}
\end{prop}

\begin{proof} (i) In view of Remark~\ref{rem:1.4}(a), 
$\check \alpha_i$ is a linear multiple of $\alpha_i^\sharp$. 
Hence (i) follows from the injectivity of the linear map 
$\sharp$ on $\Spann \Delta$. 

(ii) That $A_\Pi$ is a generalized Cartan matrix follows from 
Proposition~\ref{prop:introot}(i) and the symmetrizability follows 
from the symmetry of the matrix with entries $(\alpha_i,\alpha_j)$ 
and Remark~\ref{rem:1.4}(a). 

(iii) The choice of the $h_i$ implies that $(\fh_\Pi, \Pi, \check \Pi)$ 
is a realization of the generalized Cartan matrix $A_\Pi$. 
Let $\g(A_\Pi)$ be the corresponding Kac--Moody algebra. 
Since $A_\Pi$ is symmetrizable, the 
Gabber--Kac Theorem \cite[Thm.~2]{GK81} 
implies that it is defined by the generators $e_1, \ldots, e_r; f_1, 
\ldots, f_r; 
h_1, \ldots, h_n$ and the Serre relations. 

For $\alpha\in \Pi$ pick 
$x_{\pm\alpha} \in \g_{\pm\alpha}$ with 
$[x_{\alpha}, x_{-\alpha}] = \check \alpha$. 
Then \cite[Prop.~II.11]{Ne00b} implies the existence 
of a unique homomorphism $\phi \: \g(A_\Pi) \to \g$ which is the identity 
on the $h_i$ and maps $e_i$ to $x_{\alpha_i}$. Then 
$\phi(\g(A_\Pi)) = \fh_\Pi + \g(\Pi)$ and it remains to see that 
$\phi$ is injective. In view of \cite[Lemma~VII.5]{Ne00b}, 
its kernel is central in $\g(A_\Pi)$, but since $\phi\res_{\fh_\Pi}$ 
is injective, it is injective. 

(iv) If $A_\Pi$ is of affine type, then $\rk A_\Pi = r-1$ implies 
$n = r +1$. If $\delta$ an isotropic root of $\g(A_\Pi)$, then 
$\K \delta = \check \Pi^\bot \cap \Spann(\Pi)$, so that 
we obtain with any element $h_n \in \fh$ with $\delta(h_n)\not=0$ 
an $n$-dimensional space $\fh_\Pi$ separating the points of $\Spann\Pi$. 
\end{proof}

\begin{lemma} \label{lem:diag-auto} Each group homomorphism 
$$ \chi \: \cQ = \Spann_\Z \Delta \to \K^\times $$
defines an automorphism $\phi_\chi \in \Aut(\g)$ by 
$\phi_\chi(x) := \chi(\alpha)x$ for $x \in \g_\alpha$. 

If, conversely, there exists a rationally linearly independent subset 
$\cB \subeq \Delta_i$ with $\g = \fh + \g(\cB)$, 
then each automorphism $\phi \in \Aut(\g)$ fixing $\fh$ pointwise 
is of the form $\phi_\chi$ as above. 
\end{lemma}

\begin{proof} The first assertion is trivial. For the second, 
let $\phi \in \Aut(\g)$ fix $\fh$ pointwise. 
Then $\phi$ preserves all root spaces, 
so that there exists for each $\alpha \in \Delta_i$ a number 
$\lambda_\alpha \in \K^\times$ with 
$\phi(x_\alpha) = \lambda_\alpha x_\alpha$ for $x_\alpha \in \g_\alpha$. 
Since $\cB$ is linearly independent over $\Q$, it generates a free
subgroup of $\cQ$ and there exists a group homomorphism 
$\chi \: \cQ \to \K^\times$ with $\chi(\alpha) = \lambda_\alpha$ 
for each $\alpha \in \cB$. Now 
$\phi_\chi^{-1} \circ \phi \in \Aut(\g)$ fixes $\fh$ pointwise 
and likewise all subalgebras $\g(\alpha)$, $\alpha \in \cB$. 
Therefore it also fixes $\fh + \g(\cB) = \g$ pointwise, i.e., 
$\phi = \phi_\chi$. 
\end{proof}

\begin{prob} Let $\g$ be a coral split Lie algebra for which 
  $\Delta_i$ is connected. 
Does (2-Aff) imply the existence of an invariant symmetric bilinear form 
$\kappa$ on the commutator algebra $[\g,\g]$?  
If such a form exists and extends to a non-degenerate form on $\g$, 
then $(\g,\fh,\kappa)$ would be a split quadratic Lie algebra. 

A necessary condition for that is that the matrix 
$(\alpha(\check \beta))_{\alpha, \beta \in \Delta_i}$ is symmetrizable. 
In view of \cite[Prop.~2.3]{KN01}, this is the case if 
for any $3$-element set  $\{\alpha_1, \alpha_2, \alpha_3\}$ of integrable 
roots, the matrix $(\alpha_i(\check \alpha_j))_{i,j=1,\ldots,3}$ 
is symmetrizable. 
\end{prob} 

\section{Locally extended affine Lie algebras} \label{sec:2} 

Throughout this section, $(\g,\fh,\kappa)$ denotes an LEALA for which 
$\kappa$ is normalized such that $(\alpha,\alpha) \in \Q^\times_+$ for each 
integrable root. The goal of this section is to see that the existence 
of an integral weight $\lambda$ for which all roots in 
$$\Delta^\lambda := \{ \alpha \in \Delta \: \lambda(\alpha^\sharp) = 0\}$$  
are integrable implies that $\g$ is locally affine (cf.\ Section~\ref{sec:3}). 
If 
$$ \g^0(\lambda) = \fh + \sum_{\lambda(\alpha^\sharp)= 0} \g_\alpha $$
is generated by $\fh$ and its core, then this is also equivalent to the 
local finiteness of $\g^0(\lambda)$ (cf.\ Proposition~\ref{prop:loc-fin}). 
Conversely, for each locally affine algebra such weights exist, 
as we shall derive from Yoshii's description of the locally affine 
root systems. 

The following result is of central importance for the structure 
theory of LEALAs (\cite[Thm.~3.10]{MY06}): 

\begin{theorem}[Morita--Yoshii] \label{thm:2.1}
The form $(\cdot,\cdot)$ on $V := \Spann_\Q \Delta$ is positive 
semidefinite. 
\end{theorem}

In view of Lemma~\ref{lem:semdef}, this implies 
\begin{corollary} \label{cor:2.2} Each LEALA satisfies {\rm(2-Aff)}. 
For $\alpha, \beta \in \Delta_i$ with 
$(\beta,\beta) \geq (\alpha, \alpha)$, we have 
$(\beta,\beta)/(\alpha,\alpha) \in \{1,2,3,4\}$. 
\end{corollary}

\begin{prob} Let $(\g,\fh,\kappa)$ be coral split quadratic 
with $\Delta_i$ connected. Does (2-Aff) already imply that it is an 
LEALA, i.e., that all non-isotropic roots are integrable? This would be a 
nice characterization of LEALAs in terms of rank-2 subalgebras. 
\end{prob}

To proceed, we have to take a closer look at root systems. 
\begin{definition} \label{def:2.3}(cf.\ \cite{YY08})
Let $V$ be a rational vector space with a 
positive semidefinite bilinear form 
$(\cdot, \cdot)$ and $\cR \subeq V$ a subset. The triple 
$(V,\cR, (\cdot, \cdot))$ is called a 
{\it locally extended affine root system} or {\it LEARS} for short, 
if the following conditions are satisfied:
\begin{footnote}{In \cite{AA-P97} extended affine root systems are defined 
in such a way that they may also contain isotropic roots. To use \cite{YY08}, 
we follow Yoshii's approach. 
}\end{footnote}
\begin{itemize}
\item[\rm(A1)] $(\alpha,\alpha) \not=0$ for each $\alpha \in \cR$ and 
$\Spann \cR = V$. 
\item[\rm(A2)] $\la \beta, \alpha \ra := \frac{2(\alpha,\beta)}
{(\alpha,\alpha)} \in \Z$ for $\alpha, \beta \in \cR$. 
\item[\rm(A3)] $r_\alpha(\beta) := \beta - 
\la \beta, \alpha \ra \alpha \in \cR$ for $\alpha, \beta \in \cR$. 
\item[\rm(A4)] If $\cR = \cR_1 \cup \cR_2$ with $(\cR_1, \cR_2) \subeq \{0\}$, 
then either $\cR_1$ or $\cR_2$ is empty \break ($\cR$ is irreducible). 
\end{itemize}

A LEARS is said to be {\it reduced} if, in addition,  
\begin{itemize}
\item[\rm(R)] $2 \alpha \not\in \cR$ for each $\alpha \in \cR$. 
\end{itemize}

The root system $(V,\cR)$ is called {\it locally affine} (LARS), 
if, in addition, the 
following condition is satisfied: 
\begin{itemize}
\item[\rm(A5)] The subspace $V^0 := \{ v \in V \: (v,V) = \{0\}\}$ 
intersects $\Spann_\Z \cR$ in a non-trivial cyclic group. 
\end{itemize}

In view of (A1), (A5) implies that $\dim V^0 = 1$. 
If, in addition, $V$ is finite-dimensional, then $(V,\cR)$ is called {\it affine}. 
\begin{footnote}
{Although it is not obvious, this concept of an affine root system 
is consistent with Macdonald's (\cite{Mac72}). He assumes, 
instead of (A5), the properness of the action of the corresponding 
affine Weyl group. After tensoring with $\R$, condition (A5) 
is equivalent to the discreteness of the root system 
and the discreteness of the root 
system implies in his context the local finiteness of the 
associated system of affine hyperplanes, which in turn 
is equivalent to the properness of the action of the Weyl group 
(\cite[Cor.~3.5.9]{HoG04}). 
}\end{footnote}

The quotient space  $\oline V := V/V^\bot$ inherits a positive definite 
form for which the image $\oline \cR$ still satisfies (A1)-(A4) 
and is a locally finite root system in the sense 
of \cite{LN04}, which is not necessarily reduced (cf.\  \cite{MY06}). 
\end{definition}

The importance of these root systems is due to the following observation, 
which, in view of Definition~\ref{def:basic} 
and Remark~\ref{rem:1.4}, is an immediate consequence of 
the Morita--Yoshii Theorem~\ref{thm:2.1}. 

\begin{prop} If $(\g,\fh,\kappa)$ is a LEALA, then 
$(\Spann_\Q \Delta_i, \Delta_i, (\cdot,\cdot))$ 
is a reduced LEARS. 
\end{prop} 

In the following we put $V := \Spann_\Q \Delta_i.$
Generalizing the finite-dimensional case in \cite[Ch.~2]{ABGP97}, 
the structure of LEARS is described in detail in 
\cite{YY08}. In the following we shall only need very specific 
information, which we now recall. 
In the locally finite root system $\oline\Delta := \{ \oline\alpha \: 
\alpha \in \Delta_i\} \subeq \oline V$, 
we write 
$$\oline\Delta_{\rm red} := \{ \oline \alpha \: \oline\alpha 
\not\in 2\oline\Delta\} $$
for the corresponding reduced root system. 
Let $V' \subeq V$ a complementary subspace which is a {\it reflectable section}, 
i.e., $\Delta_{\rm red} := V' \cap \Delta$ contains an inverse image 
of each element of $\oline\Delta_{\rm red}$. We thus obtain a locally finite 
subsystem of $\Delta_{\rm red} \subeq \Delta$ spanning a hyperplane 
of $V$, on which the bilinear form is positive definite. 
Since $\Delta$ is reduced, we cannot hope for 
$V' \cap \Delta$ to map surjectively onto all of $\oline\Delta$  
if $\oline\Delta$ is not reduced. Therefore reflectable sections are 
optimal in the sense they intersect $\Delta$ in a maximal 
subset (cf.\ \cite[Lemma~4]{YY08}). 

In the following we identify $\oline V$ with $V'$ and write 
$\oline\alpha \in V'$ for the projection of $\alpha$ onto $V'$ along 
$V^0$. 
The Weyl group $\oline\cW$ of $\oline\Delta$ has at most $3$ 
orbits, determined by the square length 
(\cite[Prop.~4.4, Cor.~5.6]{LN04}).
\begin{footnote}
  {The most convenient normalization of the scalar product is that 
$(\alpha, \alpha) = 2$ for all long roots.} 
\end{footnote}
Accordingly, we write 
$\oline\Delta_{\rm sh}, \oline\Delta_{\rm lg}, \oline\Delta_{\rm ex}$
for the set of short roots (with minimal length), 
extralong roots (twice the length of a short root) and 
long roots (all others). 
There are no extralong roots if and only if the root system 
$\oline\Delta$ is reduced. 
Correspondingly, we obtain a disjoint decomposition 
$$ \Delta_i = \Delta_{\rm sh}\dot\cup \Delta_{\rm lg}\dot\cup 
\Delta_{\rm ex}. $$
We now write 
$\Delta_i = \bigcup_{\oalpha \in \oline\Delta} (\oalpha + S_\oalpha),$
where $S_\oalpha := \{ \beta \in V^0 \: \oalpha + \beta \in \Delta\}$ 
is a subset only depending on the $\oline\cW$-orbit of 
$\oline\alpha \in \oline\Delta$ (\cite{YY08}). 
For $\alpha$ short, we put $S := S_\oalpha$, 
for $\alpha$ long, we put $L := S_\oalpha$, and for 
$\alpha$ extralong, we put $E := S_\oalpha$, so that 
$$ \Delta_{\rm sh} = \oline\Delta_{\rm sh} + S, \quad 
\Delta_{\rm lg} = \oline\Delta_{\rm lg} + L \quad \mbox{ and } \quad 
\Delta_{\rm ex} = \oline\Delta_{\rm ex} + E. $$
Finally, we write $\Delta^0 := \Delta \cap V^0$ for 
the set of isotropic roots. 

\begin{lemma} \label{lem:2.5} Let $\g$ be a coral LEALA and 
$G := \la S, L, E \ra \subeq V^0$ denote the subgroup 
generated by $S$, $L$ and $E$. Then 
$$ \Delta \subeq (\oline\Delta\cup \{0\}) \oplus G \subeq 
V' \oplus V^0 $$
and there exists an $m \in \N$ with 
$$ m G + \Delta_i \subeq \Delta_i \quad \mbox{ and } \quad 
mG \subeq \Delta^0 \cup \{0\}. $$
\end{lemma} 

\begin{proof} First we show that $\Delta^0 \subeq G$. 
In view of \eqref{eq:rootrel} in Remark~\ref{rem:1.4}(c), 
for each $\delta \in \Delta^0$, there exists an 
$\alpha \in \Delta_i$ with $\delta + \alpha \in \Delta$. 
Then $(\delta + \alpha, \delta + \alpha) = (\alpha,\alpha) > 0$ 
implies that $\delta + \alpha$ is integrable. 
Now $\alpha, \alpha + \delta \in \oalpha + S_\oalpha$ 
shows that $\delta = (\alpha +\delta) - \alpha \in G$. 
This proves the first assertion. 

For the second, we write 
$k := (\beta,\beta)/(\alpha,\alpha) \in \{2,3\}$ 
for $\beta \in \Delta_{\rm lg}$ and 
$\alpha \in \Delta_{\rm sh}$ (cf.\ Proposition~\ref{prop:quotient}). 
Then 
$$ S + 2G \subeq S, \quad L + k S \subeq L\quad \mbox{ and }\quad 
 E + 4 S \subeq E$$
(cf.\ \cite{YY08}, \cite{AA-P97}) imply the existence of an 
$m \in \N$ with 
$m G + S_\oalpha \subeq S_\oalpha$ for each $\alpha\in\Delta_i$ 
and hence that 
$m G + \Delta_i \subeq \Delta_i.$

If $\alpha \in \Delta_i$ and $\delta \in G \setminus \{0\}$ satisfy 
$\beta := \alpha + \delta \in \Delta$, then 
$$\beta(\check \alpha) = \frac{2(\beta,\alpha)}{(\alpha,\alpha)} 
= \frac{2(\alpha,\alpha)}{(\alpha,\alpha)} = 2 > 0 $$
leads to $\delta = \beta - \alpha \in \Delta$ 
(cf.\ \cite[Prop.~I.7]{Ne00b}). 
In particular, we see that $mG\subeq \Delta^0\cup \{0\}$. 
\end{proof}

\begin{lemma} \label{lem:cent} If $\g = \fh + \g_c$, then 
$\alpha \in V$ is contained in $V^0$ if and only if 
$\alpha^\sharp \in \fz(\g)$. 
\end{lemma}

\begin{proof} 
In fact, $\alpha  \in V^0$ is equivalent to $\alpha \bot \Delta_i$, 
which in turn is equivalent to $\Delta_i(\alpha^\sharp) = \{0\}$, 
and hence to $\alpha^\sharp \in \fz(\g)$ (Remark~\ref{rem:1.4}(c)). 
\end{proof} 

\begin{theorem} \label{thm:2.8} 
Let $\g$ be a coral locally extended affine Lie algebra for which 
there exists an integral 
weight $\lambda \in \fh^*$ with 
$\Delta^\lambda = \{ \alpha \in \Delta \: \lambda(\alpha^\sharp) = 0\}
\subeq \Delta_i$. Then $(V,\Delta_i,(\cdot,\cdot))$ 
is a locally affine or a locally finite root system. 
\end{theorem}

\begin{proof} Since $V = \Spann_\Q \Delta_i$, Lemma~\ref{lem:2.5} 
implies that $G$ spans $V^0$. Therefore it suffices to show that 
$G$ is cyclic. In fact, Lemma~\ref{lem:2.5} yields 
$\Spann_\Z \Delta \subeq V' \oplus G$, so that the group 
$(\Spann_\Z \Delta) \cap V^0 \subeq G$ is also cyclic if $G$ has
this property. 

For $\delta \in mG$, $\alpha \in \Delta_i$ and $\beta := \alpha + \delta$ 
we have 
$\beta^\sharp = \alpha^\sharp + \delta^\sharp,$
and since $\delta^\sharp$ is central (Lemma~\ref{lem:cent}), 
$\beta(\beta^\sharp) = \beta(\alpha^\sharp) = (\beta, \alpha) 
=(\alpha,\alpha),$
which leads to 
$$ \check\beta = \frac{2}{(\alpha,\alpha)}(\alpha^\sharp + \delta^\sharp)
= \check \alpha + \frac{2}{(\alpha,\alpha)}\delta^\sharp. $$

In view of Corollary~\ref{cor:2.2}, the relation 
$\frac{2}{(\alpha,\alpha)} \lambda(\delta^\sharp) \in \Z$ 
for each integrable root $\alpha$ 
implies that 
$\lambda(G^\sharp)$ is contained in a cyclic subgroup of $\Q$. 
The condition $\Delta^\lambda \subeq \Delta_i$ is equivalent to 
$\ker \lambda \cap (\Delta^0)^\sharp = \eset,$
so that Lemma~\ref{lem:2.5} 
implies that $\lambda\res_{mG^\sharp} \: mG^\sharp \to \Q$ is 
injective with cyclic image. Therefore $G$ is a cyclic group. 

If $r := \rk \big((\Spann_\Z \Delta) \cap V^0\big)$, 
then the preceding argument implies that the 
root system $\Delta$ is either 
locally finite (for $r = 0$) or locally affine (for $r = 1$). 
\end{proof}

For the classification of locally finite and locally affine 
root systems, we refer to the appendix (Theorem~\ref{thm:yclass}). 

\begin{definition} In the following we call an integral weight 
$\lambda \in \fh^*$ of {\it transversal type} if 
$\Delta^\lambda \subeq \Delta_i$, because this is equivalent to 
$\ker \lambda \cap (V^0)^\sharp = \eset,$
i.e., $\ker \lambda$ is a hyperplane of $\fh$ transversal 
to $(V^0)^\sharp$. 
\end{definition}

The following proposition sheds some extra light on the condition 
$\Delta^\lambda \subeq \Delta_i$ in terms of the structure of the 
split quadratic Lie algebra $(\g^0(\lambda),\fh,\kappa)$. 

\begin{prop} \label{prop:loc-fin} For a coral 
quadratic split Lie algebra $(\g,\fh,\kappa)$, the following are equivalent: 
\begin{itemize}
\item[\rm(i)] $\g$ is locally finite.  
\item[\rm(ii)] $\Delta = \Delta_i$. 
\end{itemize}
\end{prop}

\begin{proof} From \cite[Thm.~VI.3]{Ne00b} we know that (ii) implies (i). 
If $\g$ is locally finite, then the Levi decomposition of 
locally finite split Lie algebras (\cite[Thm.~III.16]{St99}) 
shows that $\g_c = (\Spann\check \Delta_i) 
+ \sum_{\alpha \in \Delta_i} \g_\alpha$, 
so that all roots of $\g = \fh + \g_c$ are integrable.  
\end{proof}

\begin{remark} (a) The four dimensional split oscillator algebra $\osc$ is 
a $\K$-Lie algebra with basis 
$h,c,p,q$, where $c$ is central, 
$[p,q] = c$ and $[h,p] = p$, $[h,q] = -q$. 
Then $\fh = \K c + \K h$ is a splitting Cartan subalgebra, 
and for $\alpha(c) = 0$, $\alpha(h) = 1$ we have 
$\Delta = \{\pm\alpha\}$ and $\Delta_i = \eset$. 
In particular, $\osc$ is not coral, so that the corality is 
necessary for the implication (i) $\Rarrow$ (ii) 
in Proposition~\ref{prop:loc-fin}. 

(b) If $\g$ is a LEALA and $\oline\lambda$ an integral weight of 
$\oline\Delta$, then $\alpha^\sharp \mapsto \oline\lambda(\oline\alpha^\sharp)$ 
defines a linear functional on $V^\sharp \subeq \fh$ which we extend 
to a linear functional $\lambda$ on all of $\fh$. For each integrable 
root $\alpha$ we then have 
$$ \lambda(\check \alpha) 
= \frac{2}{(\alpha,\alpha)} \lambda(\alpha^\sharp)
= \frac{2}{(\oline\alpha,\oline\alpha)} \oline\lambda(\oline\alpha^\sharp)
= \oline\lambda(\check{\oline\alpha}) \in \Z. $$
Therefore $\lambda$ is integral, but $\lambda$ vanishes on the 
center, so that $\Delta^\lambda$ contains non-integrable roots. 
\end{remark}

\section{Locally affine Lie algebras} \label{sec:3} 

In the preceding section we have seen that the existence 
of an integral weight $\lambda$ for which all roots in 
$\g^0(\lambda)$ are integrable implies for a LEALA that its root 
system is locally affine or locally finite. 
This leads to a natural concept of a 
locally affine Lie algebra, and this section is dedicated to a 
discussion of the structure of these Lie algebras. 

The first main results in this section describes how locally affine 
Lie algebras can be described as direct limits of affine Kac--Moody 
Lie algebras. Based on this information, we then show that 
isomorphisms of locally affine root systems ``extend'' to isomorphisms 
of the corresponding minimal locally affine Lie algebras. 
From that it also follows that for a minimal locally affine 
Lie algebra, any two Cartan subalgebras 
with isomorphic root systems are conjugate under an automorphism. 

\begin{definition} \label{def:locafflie} 
We call a LEALA $(\g,\fh,\kappa)$ satisfying 
$\g = \fh + \g_c$ a {\it coral locally affine Lie algebra} 
if $\Delta_i$ is a locally affine root 
system (in its rational span) and $\Delta \not=\Delta_i$.  
\end{definition}

The following lemma helps to translate between the rational vector space 
generated by the integrable roots and its $\K$-span. 

\begin{lemma} \label{lem:3.2} If $\g$ is a coral locally affine Lie algebra 
and $V := \Spann_\Q \Delta_i$, 
then the canonical map $V \otimes_\Q \K \to \fh^*, (v, \lambda) \mapsto 
\lambda v$ is a linear isomorphism onto $\Spann_\K \Delta$. 
\end{lemma} 

\begin{proof} We have to show that if 
$\alpha_0, \ldots, \alpha_n \in V$ are linearly independent over 
$\Q$, then they are also linearly independent over $\K$. We may 
assume that $\alpha_0$ is contained in the one-dimensional space $V^0$. 
Since $(\cdot,\cdot)$ is non-degenerate modulo $V^0$, there 
exist $\alpha_1^*, \ldots, \alpha_n^* \in V$ with 
$(\alpha_i, \alpha_j^*) = \delta_{ij}$ for $i,j=1,\ldots, n$. 

Suppose that $\sum_{i = 0}^n t_i \alpha_i = 0$ for $t_i \in \K$. 
Then 
$0 = \big(\sum_{i = 1}^n t_i \alpha_i, \alpha_j^*\big) = t_j$ 
for $j =1,\ldots, n,$
and hence $t_0 \alpha_0 = 0$. As $\alpha_0$ is non-zero, 
it also follows that $t_0 = 0$. This proves the lemma. 
\end{proof}

\begin{prop} \label{prop:2.11} 
For a coral locally affine Lie algebra $\g$, a generator 
$\delta$ of the group $V^0\cap \Spann_\Z \Delta_i$ and an element 
$h_0 \in \fh$ with $\delta(h_0) \not=0$, the following assertions hold: 
\begin{itemize}
\item[\rm(i)] $\dim(\fz(\g_c)) = 1$. 
\item[\rm(ii)] $\Delta_i$ is the directed union of all finite connected 
subsets $F$ with $\delta \in \Spann_\Z F$. For each such $F$, 
the following assertions hold for $V_F := \Spann_\Q F$: 
\begin{itemize}
\item[\rm(a)] If $\Delta_i^F := V_F \cap \Delta_i$, then 
$(\Delta_i^F, V_F)$ is an affine root system. 
\item[\rm(b)] $\Delta^F_i$ contains a linearly independent 
simple system $\Pi_F$, i.e., 
$\alpha - \beta \not\in \Delta$ for $\alpha,\beta \in \Pi_F$. 
\item[\rm(c)] The subalgebra $\g(\Pi_F) + \K h_0$ is isomorphic to the 
affine Kac--Moody algebra  $\g(A_{\Pi_F})$ and its root system 
is $\Delta^F = \Delta_i^F \cup (\Z \setminus \{0\})\delta.$
\end{itemize}
\item[\rm(iii)] $\delta \in \Delta$ and 
$\Delta^0 = \Z\delta \setminus \{0\}$ is the set of isotropic roots. 
\item[\rm(iv)] $\Delta_i$ contains a 
linearly independent subset 
$\cB$ with $\Delta \subeq \Spann_\Z \cB$. 
We call such a set an {\em integral base of $\Delta$}. 
\end{itemize}
\end{prop}

\begin{proof} (i) First we recall from Remark~\ref{rem:1.4}(c) that 
$\Delta^\sharp \subeq \fh \cap \g_c = \Spann_\K \check \Delta_i,$
so that 
$\sharp \: \Spann_\K \Delta \to \fh \cap \g_c$
is a linear isomorphism. Next we observe that, for 
$\alpha \in \Spann_\K \Delta$,  the relation $\alpha \in V^0$ 
is equivalent to $\alpha^\sharp \in \fz(\g_c)$ (Lemma~\ref{lem:cent}). 
Now (i) follows from $\dim V^0 = 1$. 

(ii), (iii) If $M \subeq \Delta_i$ is a finite subset, then the connectedness 
of $\Delta_i$ implies the existence of a finite connected subset 
$\tilde M \subeq \Delta_i$ containing $M$. Since 
$\delta \in \Spann_\Z \Delta_i$, it follows that $\Delta_i$ 
is the directed union of all finite connected subsets $F$ with 
$\delta \in \Spann_\Z F$. 

Clearly, $(V_F, \Delta_i^F)$ satisfies (A1)-(A3). Since 
$V_F$ is spanned by a connected set of roots, $\Delta_i^F$ is 
irreducible, so that it is an affine root system. 
Moreover, $\Delta_i^F$ is discrete in $V_F \otimes_\Q \R$ because 
its image in $\oline V$ is finite (cf.\ \cite[Lemma~2.8]{AA-P97}) 
and the fibers of the map $\Delta \to \oline\Delta$, i.e., 
the sets $S_{\oline\alpha}$ are contained in a cyclic group. 
That each affine 
root system contains a linearly independent 
simple system follows from the discussion in (\cite[Sect~2]{ABGP97}). 

Next we use Proposition~\ref{prop:1.10}(iv) to see that 
$\g(\Pi_F) + \K h_0$ is isomorphic to the affine Kac--Moody 
algebra $\g(A_{\Pi_F})$. From \cite[\S 5.5]{Ka90} we know that 
$\Z \delta \setminus \{0\} \subeq \Delta^F \subeq \Delta$, 
and since $V^0 \cap \Spann_\Z \Delta = \Z \delta$, it follows 
that $\Delta^0 = \Z \delta \setminus \{0\}$. This completes 
the proof of (ii), and (iii) also follows. 

(iv) First we recall from 
\cite[Thm.~VI.6]{St99} that the locally finite root system 
$\oline\Delta$ has an integral base $\oline\cB$ 
(see also \cite[Cor.~6.5]{LN04}). 
Let $\cB_1 \subeq \Delta$ be a subset mapping bijectively 
onto $\oline\cB$ and observe that this implies that 
$(\cdot,\cdot)$ is positive definite on $\Spann_\Q \cB_1$. 

For $V_F$ as above, we may w.l.o.g.\ assume that 
$\Delta^F$ contains an element $\alpha_1 \in \cB_1$. 
This element is part of a simple system of the affine root system 
$\Delta^F$, so that $\alpha_0 := \delta - \alpha_1 \in \Delta^F_i$. 
Then we put 
$\cB := \cB_1 \cup \{\alpha_0\}$ (cf.\ Proposition~\ref{prop:2.11}). 
Since $\oline\cB$ is linearly 
independent and $\delta \in V^0$, the subset $\cB \subeq V$ is also 
linearly independent. 

To see that $\Delta \subeq \Spann_\Z \cB = \Spann_\Z \cB_1 + \Z \delta$, let 
$\beta \in \Delta$. Since $\oline\cB$ is an integral base 
of $\oline\Delta$, there exist $\alpha_1, \ldots, \alpha_N \in \cB_1$ such 
and $n_i \in \Z$ such that 
$$\beta- \sum_{j=1}^N n_j \alpha_j \in V^0 \cap \Spann_\Z \Delta 
= \Z \delta. $$
This proves that $\beta \in \Spann_\Z \cB$. 
\end{proof}

\begin{prop} \label{prop:3.4} If $\g$ is a coral locally affine Lie algebra, 
then for an integral weight $\lambda \in \fh^*$, the condition 
$\Delta^\lambda \subeq \Delta_i$ is equivalent to 
$\lambda\res_{\fz(\g_c)} \not=0$. Such weights exist. 
\end{prop}

\begin{proof} Since $\Delta\setminus \Delta_i = \Z \delta \setminus \{0\}$ 
for an isotropic root $\delta$ (Proposition~\ref{prop:2.11}), 
$\Delta^\lambda \subeq \Delta_i$
is equivalent to $\lambda(\delta^\sharp) =0$. 
As $\delta^\sharp$ generates $\z(\g_c)$ 
(Lemma~\ref{lem:cent}), 
this in turn is equivalent to $\lambda\res_{\z(\g_c)} \not=0$. 

That weights of transversal type actually exist 
can be derived from the description of the affine root system 
$\Delta$ in terms of the locally finite subsystem 
$\Delta_{\rm red}$. Since the assertion is trivial for the finite-dimensional 
case, we may assume that $\dim V = \infty$. 
From Yoshii's classification in Theorem~\ref{thm:yclass} below,  
it follows that 
$$ \Delta \subeq 
(\Delta_{\rm red} + \Z \delta)\cup 
\big(2(\Delta_{\rm red})_{\rm sh} + (2\Z+1) \delta\big). $$
From this information one can easily calculate the possible coroots. 
For a root of the form $n\alpha + m \delta$, $\alpha \in \Delta_{\rm red}$, 
$n \in \{1,2\}$, the corresponding coroot is determined by 
$(n\alpha + m \delta)\,\check{} \in \Q (n\alpha^\sharp +  m\delta^\sharp)$ and 
$(n\alpha+ m \delta)((n\alpha + m \delta)\,\check{})=2$, which leads to 
\begin{equation}
  \label{eq:mul-coroot}
(n\alpha + m \delta)\,\check{} = \frac{2}{n^2(\alpha,\alpha)}(n\alpha^\sharp 
+ m \delta^\sharp), \quad n \in \{1,2\}, m \in \Z. 
\end{equation}
Therefore a linear functional $\lambda \in \fh^*$ vanishing on 
$\check\Delta_{\rm red}$ is integral if 
$$ \lambda(\delta^\sharp) \in 2(\alpha,\alpha)\Z $$
holds for each integrable root $\alpha$. In this case we have 
$\Delta^\lambda = \Delta_{\rm red}$.  
Since at most three square lengths occur (Proposition~\ref{prop:quotient}), 
this proves the existence of integral weights $\lambda$ 
of transversal type. 
\end{proof}

\begin{remark} \label{rem:3.4} (i) If $\Delta = \Delta_i$, 
then $\g$ is locally finite (\cite[Thm.~VI.3]{Ne00b}). 
If, in addition, $\g$ is perfect and $\Delta_i$ connected, 
then $\g$ carries the structure of a 
locally extended affine Lie algebra for which $(\cdot,\cdot)$ is 
positive definite on $\Spann_\Q \Delta$ (\cite[Thm.~4.2]{LN04}). In this case 
$\Delta^\lambda \subeq \Delta_i$ trivially holds for any integral 
weight~$\lambda$. 

(ii) Suppose that $\g$ is affine and that $\Pi = \{\alpha_1, \ldots, 
\alpha_r\}\subeq \Delta$ is a 
generating linearly independent simple system. 
Let $\lambda \in \fh^*$ be a dominant integral weight 
not vanishing on all coroots, i.e., 
$$\Pi_\lambda := \{ \alpha \in \Pi \: \lambda(\check \alpha) = 0 \} 
\not=\Pi.$$ 
We claim that $\lambda$ does not vanish on the center. 
Using the notation of \cite{Ka90}, we write a generator of the 
center as $K = \sum_{j = 1}^r \check a_j \check \alpha_j$, where all 
coefficients $\check a_j$ are positive. Then  
$$\lambda(K) = \sum_{\alpha_i \not\in \Pi_\lambda} \check a_{i} 
\lambda(\check\alpha_{i}) > 0. $$
\end{remark}  

\subsection*{Minimal locally affine Lie algebras}

The following notion of minimality distinguishes a class of 
locally affine Lie algebras which, as we shall see, are 
uniquely determined by their root systems. 

\begin{definition} \label{def:min} 
We call a locally affine Lie algebra $(\g,\fh,\kappa)$ 
{\it minimal} if $\g_c$ is a hyperplane in $\g$ and there exists an element 
$d \in \fh$ for which 
$\{ \alpha \in \Delta_i \: \alpha(d)=0\}$ 
is a reflectable section. Then $\delta(d) \not=0$ and we may 
normalize $d$ by $\delta(d) = 1$ (cf.\ Proposition~\ref{prop:2.11}). 
\end{definition}

To analyze how minimal locally affine Lie algebras can be reconstructed 
from their core, we need the concept of a double extension 
of a quadratic Lie algebra (cf.\ \cite{MR85}). 

\begin{definition} \label{def:doubext} Let $(\oline\g,\oline\kappa)$ 
be a  quadratic Lie algebra and $D \in \der(\oline\g,\oline\kappa)$ 
be a derivation which  is skew-symmetric with respect to $\oline\kappa$. 
Then $\omega_D(x,y) := \oline\kappa(Dx,y)$ defines a $2$-cocycle 
on $\oline\g$ and $D$ extends to a derivation $\tilde D(z,x) := (0,Dx)$ 
of the corresponding central extension 
$\K \oplus_{\omega_D} \oline\g$. The Lie algebra 
$$ \g= (\K \oplus_{\omega_D} \oline\g) \rtimes_{\tilde D} \K$$ 
with the Lie bracket 
$$ [(z,x,t), (z',x',t')] = (\omega_D(x,x'), [x,x'] + tDx' - t'Dx,0) $$
is called the corresponding {\it double extension}. 
It carries a non-degenerate invariant symmetric bilinear form 
$$ \kappa((z,x,t), (z',x',t)) = \oline\kappa(x,x') + zt'+z't, $$
so that $(\g,\kappa)$ also is a quadratic Lie algebra. 
\end{definition}

\begin{remark} \label{rem:3.7} (a) Each affine Kac--Moody algebra 
is a minimal locally affine Lie algebra. 

(b) If $\g_c$ is the core of a locally affine Lie algebra, 
then $\g_c$ is graded by the root group $\cQ = \Spann_\Z \Delta$. 
Let $\lambda \in \fh^*$ not vanish on the center and pick 
$c \in \fz(\g_c)$ with $\lambda(c) = 1$. We extend $\lambda$ 
to a linear functional, also called $\lambda$, on $\g$, vanishing on all 
root spaces. 
Then $\ker \lambda \subeq \g_c$ is a subspace mapped bijectively 
onto the centerless core $\g_{cc} := \g_c/\fz(\g_c)$. 
For $x,x' \in \ker \lambda$ and $a,a' \in \K$ we have 
$$ [x + ac, x' + ac'] = [x,x'] = ([x,x'] - \lambda([x,x'])c) + \lambda([x,x'])c, $$
so that $\g_c$ is the central extension 
$\K \oplus_\omega \g_{cc}$ 
defined by the cocycle $\omega(\oline x,\oline x') = \lambda([x,x'])$, 
where $\oline x := x + \fz(\g_c)$. 

Next we observe that $\lambda$ defines a diagonal derivation 
$\tilde D_\lambda \in \der(\g_c)$ by 
$$ \tilde D_\lambda x = \lambda(\alpha^\sharp) x \quad \mbox{ for } \quad 
x \in \g_\alpha. $$
This derivation also induces a derivation $D$ on 
the centerless core $\g_{cc} = \g_c/\fz(\g_c)$ preserving the 
induced non-degenerate symmetric bilinear form $\oline\kappa$. 

(c) Let $\hat\g= (\K \oplus_\omega \g_{cc}) \rtimes_D \K$ be the corresponding double 
extension (Definition~\ref{def:doubext}). 
With $\fh_c := \fh \cap \g_c$ we now obtain a splitting 
Cartan subalgebra $\hat\fh := \R \oplus \fh_c \oplus \R$ of $\hat\g$, 
for which the corresponding root decomposition coincides with the 
$\cQ$-grading. In particular, we obtain a realization of the locally 
affine root system $\Delta$ in $\fh^*$. 

To obtain a minimal locally affine Lie algebra $\hat\g$ with this procedure,
we have to assume, in addition, that 
$\Delta^\lambda = \{ \alpha \in \Delta \: \lambda(\check \alpha) = 0\} 
= \Delta_{\rm red}$ 
holds for a reflectable section. Since such functionals exist 
by Proposition~\ref{prop:3.4}, we derive the 
existence of a minimal realization for $\Delta$. We refer to the 
appendix for minimal realizations of the infinite rank affine root systems 
by twisted loop algebras. 

(d) Let us assume, in addition, 
that there exists an element $d\in \fh$ with 
$\lambda = d^\flat$. Then 
$\lambda(\alpha^\sharp) = \kappa(d,\alpha^\sharp) = \alpha(d)$ 
implies that $\tilde D = \ad d$ and 
$$ \omega_D(\oline x, \oline x') = \kappa([d,x],x') 
= \kappa(d,[x,x']) = \lambda([x,x']). $$
We also note that for $x,x' \in \ker\lambda \cap \g_c$ and 
$a,a', b,b,' \in \K$ we have 
$$ \kappa(a c + x + bd, a'c + x' + b'd) 
=  \kappa(x,x') + ba'+ab', $$
which shows that with $\fh_c := \fh \cap \g_c$ we obtain an isomorphism 
$$(\hat\g, \hat\fh,\hat\kappa) \cong (\g_c + \K d, \fh_c + \K d, \kappa)$$ 
of split 
quadratic Lie algebras. We thus find a minimal locally affine 
subalgebra of~$\g$. 
\end{remark}

\subsection*{The Extension Theorem} 

\begin{definition} Two locally extended affine root 
systems $(V_1, \cR_1)$ and $(V_2, \cR_2)$
are said to be isomorphic if there exists a linear isomorphism 
$\psi \: V_1 \to V_2$ with $\psi(\cR_1) = \cR_2$. 
\end{definition}

Since the quadratic form is part of the concept of a LEARS, it should 
also be taken into account for the concept of an isomorphism, but 
the following lemma shows that this is redundant (\cite[Lemma~9]{YY08}): 

\begin{lemma} If $\phi \: (V_1, \cR_1) \to (V_2, \cR_2)$ is an isomorphism 
of locally extended affine root systems, then 
\begin{equation}
  \label{eq:co-rel}
\la \phi(\alpha), \phi(\beta) \ra = \la \alpha, \beta\ra \quad \mbox{ 
for } \quad \alpha, \beta \in \cR_1, 
\end{equation}
$\phi$ preserves the quadratic form up to a factor, and 
$\phi \circ r_\alpha \circ \phi^{-1} = r_{\phi(\alpha)}$ for 
$\alpha \in \cR_1.$
\end{lemma}

Together with Proposition~\ref{prop:2.11}, the following lemma is 
the key ingredient in our Extension Theorem. It provides the required 
local information. 

\begin{lemma}\label{lem:3.3}  
Let $(\g_1, \fh_1)$ and $(\g_2, \fh_2)$ be affine Kac--Moody 
Lie algebras and 
$\psi \: \Delta_1 \to \Delta_2$ an isomorphism 
of affine root systems. Further, let 
$\cB \subeq \Delta_{1,i}$ be an integral 
base and pick $0\not=x_\alpha \in \g_{1,\alpha}$ 
and $0\not=y_\alpha\in \g_{2, \psi(\alpha)}$ for $\alpha \in \cB$. 
Then there exists a unique isomorphism of 
Lie algebras 
$$ \phi \: (\g_1)_c \to (\g_2)_c \quad \mbox{ with } \quad 
\phi(\g_{1,\alpha}) = \g_{2,\psi(\alpha)} \quad \mbox{ for } 
\quad \alpha \in \Delta_1 $$
and 
$\phi(x_\alpha) = y_\alpha$ for $\alpha \in \cB.$
\end{lemma}  

\begin{proof} Let $\Pi_1 \subeq \Delta_1$ be a linearly independent 
generating simple system and $\Pi_2 := \psi(\Pi_1)$. 
In view of \eqref{eq:co-rel}, 
$$ \psi(\alpha)(\psi(\beta)\,\check{}) 
= \la \psi(\alpha), \psi(\beta)\ra 
= \la \alpha, \beta\ra 
=  \alpha(\check \beta)\quad \mbox{ for } \quad \alpha, \beta \in \Pi_1, $$
so that $\g_1$ and $\g_2$ correspond to the same generalized 
Cartan matrix, hence are isomorphic (cf.\ \cite[Ch.~1]{Ka90}). 

Let $\gamma \: \g_1 \to \g_2$ be an isomorphism with 
$\gamma(\fh_1) = \fh_2$, inducing the 
isomorphism $\phi \: \Delta_1 \to \Delta_2$. Since the 
non-isotropic root spaces $\g_{2,\psi(\alpha)}$ are $1$-dimensional, 
there exist scalars $\lambda_\alpha \in \K^\times$ with 
$$ \gamma(x_\alpha) = \lambda_\alpha y_\alpha \quad \mbox{ for } \quad 
\alpha \in \cB. $$
Since $\cB \subeq \Delta_1$ is linearly independent, 
there exists a group homomorphism  
$$ \chi \: \Spann_\Z \cB \to \K^\times \quad \mbox{ with } \quad 
\chi(\alpha) = \lambda_\alpha\quad \mbox{ for } \quad \alpha \in \cB. $$
Then 
$\phi_\chi(x) := \chi(\alpha) x$ for $x \in \g_{1,\alpha}$ 
defines an automorphism of Lie algebras (Lemma~\ref{lem:diag-auto}) 
and 
$$ \phi := \gamma \circ \phi_\chi^{-1} \: \g_1 \to \g_2 $$
maps each $x_\alpha$, $\alpha \in \cB$, to the corresponding 
element $y_\alpha \in \g_2$. This proves the existence of $\phi$. 

For the uniqueness, we assume that $\tilde\phi \: (\g_1)_c \to (\g_2)_c$ 
is another isomorphism with the same properties. Then 
$\Phi := \tilde\phi^{-1} \circ \phi \: (\g_1)_c \to (\g_1)_c$ is an 
isomorphism preserving each root space and fixing each 
$x_\alpha$, $\alpha \in \cB$. We have to show that this implies 
that $\Phi = \id_{\g_{1,c}}$. 

On each $3$-dimensional subalgebra $\g(\beta)$, $\beta \in \Delta_{1,i}$, 
$\Phi$ induces an automorphism preserving the root decomposition. 
This implies that $\Phi(\check\beta) = \check\beta$ and that 
$\Phi(x_\beta) = \mu_\beta x_\beta$ for some $\mu_\beta \in \K^\times$. 
Let $\Pi_1 = \{ \alpha_1,\ldots, \alpha_r\} 
\subeq \Delta_1$ be a generating simple system 
and $\mu_j := \mu_{{\alpha_j}}$. 
Let $\nu \: \Spann_\Z \Pi_1 \to \K^\times$ be the unique group 
homomorphism mapping $\alpha_j$ to $\mu_j$. 
Then $\phi_\nu \in \Aut(\g_1)$ is the unique automorphism 
fixing $\fh_1$ pointwise and multiplying each $x_{\alpha_j}$ with 
$\mu_j$. We conclude that $\Phi = \phi_\nu$, which implies that 
$\nu(\alpha) = 1$ for each $\alpha \in \cB$. Now 
$\Delta \subeq \Spann_\Z \cB$ finally leads to $\nu =1$, so that 
$\Phi = \phi_\nu = \id_{\g_{1,c}}$. 
\end{proof}

\begin{theorem}[Extension Theorem] \label{thm:3.3}  
Let $(\g_1, \fh_1, \kappa_1)$ and $(\g_2, \fh_2, \kappa_2)$ 
be locally affine Lie algebras.  
If 
$\psi \: (V_1, \Delta_1) \to (V_2, \Delta_2)$ is an isomorphism 
of locally affine root systems, then there exists an isomorphism of 
Lie algebras 
$$ \phi \: (\g_1)_c \to (\g_2)_c \quad \mbox{ with } \quad 
\phi(\g_{1,\alpha}) = \g_{2,\psi(\alpha)} \quad \mbox{ for } \quad 
\alpha \in \Delta_1. $$

If $\cB \subeq \Delta_{1,i}$ is an integral base 
$0\not=x_\alpha \in \g_{1,\alpha}$, $0\not=y_\alpha\in 
\g_{2, \psi(\alpha)}$ for $\alpha \in \cB$, 
then there exists a unique such $\phi$ with 
\begin{equation}
  \label{eq:unique2}
\phi(x_\alpha) = y_\alpha \quad \mbox{ for } \quad \alpha \in \cB. 
\end{equation}
\end{theorem} 

\begin{proof} Let $\cB \subeq \Delta_{1,i}$ be an integral base 
(Proposition~\ref{prop:2.11}(iv)). 
For each $\alpha \in \cB$, we pick non-zero elements 
$x_\alpha \in \g_{1,\alpha}$ and $y_\alpha \in \g_{2,\psi(\alpha)}$.  

Let $F \subeq \cB$ be a connected finite subset with 
$\delta \in \Spann_\Z F$ and $\Delta^F := \Delta \cap \Spann F$, 
so that the subalgebra $\g_1(\Delta^F_i)  \subeq \g_1$ 
is the core of an affine 
Kac--Moody algebra and $F \subeq \Delta^F$ is an integral  
base (Proposition~\ref{prop:2.11}). 
With Lemma~\ref{lem:3.3} we now obtain a unique isomorphism 
$$ \phi_F \: (\g_{1,F})_c  \to (\g_{2,\psi(F)})_c \quad \mbox{ with } 
\quad \phi_F(x_\alpha) = y_\alpha \quad \mbox{ for } \quad 
\alpha \in F. $$ 

For any larger finite subset $E \supeq F$ with the same properties, we 
likewise obtain a unique isomorphism 
$$ \phi_E \: (\g_{1,E})_c \to (\g_{2,\psi(E)})_c \quad \mbox{ with } 
\quad \phi_E(x_\alpha) = y_\alpha \quad \mbox{ for } \quad 
\alpha \in E, $$ 
and the uniqueness of $\phi_F$ implies that 
$\phi_E\res_{\g_F} = \phi_F$. 
We conclude that the isomorphisms $\phi_F$ combine to a unique 
isomorphism 
$$\phi \: \g_{1,c} = \bigcup_F (\g_{1,F})_c  
\to \g_{2,c} = \bigcup_F (\g_{2,F})_c 
\quad \mbox{ with } \quad 
\phi(x_\alpha) = y_\alpha \quad \mbox{ for } \quad 
\alpha \in \cB.$$ 
\end{proof}

\begin{corollary}\label{cor:3.3}  
If $(\g_1, \fh_1, \kappa_1)$ and $(\g_2, \fh_2, \kappa_2)$ 
are locally affine Lie algebras with isomorphic root systems, then 
their cores are isomorphic. 
\end{corollary} 

\begin{remark} In general there is no unique extension 
of $\phi$ to all of $\g_1$. If $h_1 \in \fh_1$, then any 
such extension mapping $\fh_1$ into $\fh_2$ would have to map 
$h$ to an element $h_2$ satisfying 
$\psi(\alpha)(h_2) = \alpha(h_1)$ for each 
$\alpha \in \Delta_1$. 
This determines $h_2$ uniquely up to a central element. On the other 
hand, every linear map $\g_1/[\g_1,\g_1] \to \z(\g_2)$ is a homomorphism 
of Lie algebras that can be added to any homomorphism 
$\phi \: \g_1 \to \g_2$. 
\end{remark}

\begin{theorem}[Uniqueness Theorem]\label{thm:3.4}  
If $(\g_1, \fh_1, \kappa_1)$ and $(\g_2, \fh_2, \kappa_2)$ 
are minimal locally affine Lie algebras with isometrically 
isomorphic root systems, then there exists an isomorphism 
$\phi \: (\g_1, \fh_1, \kappa_1) \to (\g_2, \fh_2, \kappa_2)$ 
of quadratic split Lie algebras. 
\end{theorem} 

\begin{proof} Since both Lie algebras $\g_j$ are minimal locally affine, 
there exist $d_j \in \fh_j$, $j =1,2$,  such that 
$$\Delta_{j,{\rm red}} := \{ \alpha \in \Delta_j \: \alpha(d_j) = 0 \} $$
define reflectable sections and $\delta_j(d_j) = 1$ holds for the respective basic 
isotropic roots $\delta_j$. In view of Theorem~\ref{thm:beam-class} below, 
there exists an isomorphism $\psi \: \Delta_{1,i} \to \Delta_{2,i}$ 
of root systems mapping $\Delta_{1,{\rm red}}$ to $\Delta_{2,{\rm red}}$, 
and we may further assume that $\psi(\delta_1) = \delta_2$ (which can be 
achieved by replacing $\psi$ by $-\psi$ if necessary). 
Then $\psi(\alpha)(d_2) = \alpha(d_1)$ holds for 
$\alpha\in \Delta_{1,{\rm red}}$ and also for $\alpha = \delta_1$, 
hence for each $\alpha \in \Delta_1$. 

Now we apply Theorem~\ref{thm:3.3} to obtain an isomorphism 
$\phi \: \g_{1,c} \to \g_{2,c}$ with 
$\phi(\g_{1,\alpha}) = \g_{2,\psi(\alpha)}$ for 
$\alpha \in \Delta_{1,i}$. For $x \in \g_{1,\alpha}$ we have
$$ \phi([d_1, x]) = \alpha(d_1)\phi(x) 
= \psi(\alpha)(d_2) \phi(x)
= [d_2, \phi(x)],$$
which implies that $\phi \circ \ad d_1 = \ad d_2 \circ \phi$. 

Next we claim that $\phi$ is isometric. To this end, we consider 
the symmetric bilinear form $\kappa := \phi^*\kappa_2 - \kappa_1$ 
on $\g_{1,c}$. To see that $\kappa$ vanishes, we note that its 
radical is an ideal of $\g_{1,c}$. If we can show that 
it contains $\check \Delta_{i,1}$, then it also contains 
all subalgebras $\g(\alpha)$, $\alpha \in \Delta_i$, and therefore 
all of $\g_{1,c}$. It therefore remains to show that 
for $\alpha, \beta \in \Delta_{1,i}$ we have 
\begin{equation}
  \label{eq:kappa-rel1}
\kappa_2(\phi(\check \alpha), \phi(\check \beta)) 
= \kappa_1(\check \alpha, \check \beta). 
\end{equation}
From $\phi(\g_{1,\alpha}) = \g_{2,\psi(\alpha)}$ we derive that 
$\phi(\check \alpha) = \psi(\alpha)\,\check{}$, so that 
$$ \kappa_1(\check \alpha, \check \beta) 
= \frac{4(\alpha,\beta)}{(\alpha,\alpha)(\beta,\beta)}
= \frac{4(\psi(\alpha),\psi(\beta))}{(\psi(\alpha),\psi(\alpha))
(\psi(\beta),\psi(\beta))}
= \kappa_2(\phi(\check \alpha), \phi(\check \beta)) $$ 
implies \eqref{eq:kappa-rel1}. 
This proves that $\phi \: \g_{1,c} \to \g_{2,c}$ is isometric. 

In particular, we derive that the induced isomorphism 
$\oline\phi \: \g_{1,cc} \to \g_{2,cc}$ of the centerless cores 
is isometric and intertwines the derivations 
$D_j$ induced by $\ad d_j$ on $\g_{j,cc}$. Therefore $\oline\phi$ 
extends to an isomorphism 
$$ \hat\phi \: \hat\g_1 \to \hat\g_2, \quad 
(z,x,t) \mapsto (z, \oline\phi(x), t) $$
of the corresponding double extensions (Definition~\ref{def:doubext}). 
Finally, Remark~\ref{rem:3.7}(d) 
implies that $\hat\g_j \cong \g_j$, and the assertion follows. 
\end{proof}

\begin{corollary}\label{cor:3.4a}  
If $(\g, \fh, \kappa)$ is a minimal locally affine Lie algebra 
and $\fh'$ is another splitting Cartan subalgebra for which 
the corresponding root system $\Delta'$ is isomorphic to $\Delta$, 
then there exists an automorphism $\phi$ of $\g$ with 
$\phi(\fh)= \fh'$. 
\end{corollary}

\subsection*{Unitary real forms} 

In this subsection, we consider complex Lie algebras and suitable real 
forms which are compatible with all the relevant structure. 
The main point is the existence of unitary real forms of minimal 
locally affine Lie algebras. 

\begin{definition} \mlabel{def:unitaryrealform} 
An {\it involution} of a complex quadratic split 
Lie algebra $(\g,\fh,\kappa)$ is an involutive antilinear antiautomorphism 
$\sigma \: \g \to \g, x \mapsto x^*$ satisfying 
\begin{itemize}
\item[\rm(I1)] $\alpha(x) \in \R$ for $x = x^* \in \fh$. 
\item[\rm(I2)] $\sigma(\g_\alpha) = \g_{-\alpha}$ for $\alpha \in \Delta.$ 
\item[\rm(I3)] $\kappa(\sigma(x),\sigma(y)) = \oline{\kappa(x,y)}$ for 
$x,y \in \g.$
\end{itemize}
Then we call $(\g,\fh,\kappa,\sigma)$  an {\it involutive quadratic split 
Lie algebra} and write 
$$\fk := \fk(\sigma) := \{ x \in \g \: x^* = - x\}$$ for the corresponding 
real form of $\g$. 
We call $\sigma$ and the corresponding real form $\fk$ {\it unitary} 
if the hermitian form 
$$ \kappa_\sigma(x,y) := \kappa(x,\sigma(y)) $$
is positive semidefinite on $\g_c$. 
\end{definition}

\begin{lemma} \label{lem:3.14} 
Let $(\g,\fh,\kappa,\sigma)$ be an involutive affine Kac--Moody algebra 
with a unitary real form. For $\alpha \in \Delta_i$, 
pick $x_{\pm\alpha} 
\in \g_{\pm\alpha}$ with $[x_\alpha, x_{-\alpha}] = \check \alpha$. Then 
$$ x_\alpha^* = \lambda_\alpha x_{-\alpha} \quad \mbox{ for some } 
\quad \lambda_\alpha > 0. $$
\end{lemma} 

\begin{proof} It is easy to verify this by computation because 
the positivity of $\lambda_\alpha$ is equivalent to 
$\fk \cap \g(\alpha) \cong \su_2(\C)$, the 
compact real form of $\fsl_2(\C) \cong \g(\alpha)$.  
\end{proof}

\begin{prop}\label{prop:ex-unitary} 
If $(\g,\fh,\kappa)$ is a complex minimal locally affine Lie algebra, then 
$\g$ has a unitary real form. 
\end{prop}

\begin{proof} Let $\oline\g$ denote the same Lie algebra $\g$, endowed 
with the complex conjugate scalar multiplication 
$z \bullet x := \oline z x$. Then $(\oline\g,\oline\fh,\oline\kappa)$ also 
is a complex minimal locally affine Lie algebra with 
$\oline\g_\alpha = \g_{\oline\alpha}$. Note that 
$\oline\alpha \: \oline\fh \to \C$ is complex linear because 
$\oline\fh$ carries the opposite complex structure.  

Now $\psi \: \Delta(\g,\fh) \to \Delta(\oline\fg,\oline\fh), \alpha 
\mapsto - \oline \alpha$ is an isometric isomorphism 
of locally affine root systems. 
With Theorem~\ref{thm:3.4} we obtain an 
isometric isomorphism 
$$ \tilde\sigma \: (\g,\kappa) \to (\oline\g,\oline\kappa) 
\quad \mbox{ with } \quad 
\tilde\sigma(\g_\alpha) = \g_{-\oline\alpha}
\quad \mbox{ and } \quad \tilde\sigma(d) = -d. $$
We next define an antilinear map 
$\sigma \: \g \to \g, x \mapsto -\tilde\sigma(x)$ and 
obtain an involutive antiautomorphism of 
$\g_c$ satisfying 
$\sigma(\g_\alpha) =\g_{-\alpha}$ for each $\alpha \in \Delta$. 

Let $\cB \subeq \Delta_i$ be an integral base and 
pick $x_{\pm\alpha} \in \g_{\pm\alpha}$, $\alpha \in \cB$, in such a 
way that $[x_\alpha, x_{-\alpha}] = \check \alpha$. 
In view of Theorem~\ref{thm:3.3} (cf.\ also Lemma~\ref{lem:diag-auto}), 
we may 
choose $\sigma$ in such a way that $\sigma(x_\alpha) =  x_{-\alpha}$. 
Then $\sigma^2(x_\alpha) = x_\alpha$ for each $\alpha \in \cB$, so that 
the uniqueness assertion in Theorem~\ref{thm:3.3} 
implies that $\sigma^2 = \id$ on $\g_c$, so that $\sigma(d) = d$ 
leads to $\sigma^2 = \id_\g$, i.e., $\sigma$ defines a real 
form $\fk := \{ x \in \g \: \sigma(x) = x\}$.  

To see that $\fk$ is unitary, we first show that 
whenever $x_{\pm\alpha} \in \g_{\pm\alpha}$ satisfy 
$[x_\alpha, x_{-\alpha}] = \check \alpha$, then 
$\sigma(x_\alpha) = \lambda_\alpha x_{-\alpha}$ for some 
real $\lambda_\alpha > 0$ (Lemma~\ref{lem:3.14}). 
Let $V_F$ be as in Proposition~\ref{prop:2.11}(iv) and $\g_F \subeq \g$ 
be a corresponding affine Kac--Moody subalgebra. Then 
$\sigma(\g_{F,c}) = \g_{F,c}$, and $\sigma$ induces an involution 
on the core $\g_{F,c}$ of $\g_F$. We know from 
\cite[\S 2.7 and  Thm.~11.7]{Ka90} that there exists a unitary involution 
$\sigma_c$ on $\g_F$. We thus obtain a complex linear automorphism 
$$ \phi := \sigma \circ \sigma_c \: \g_{F,c} \to \g_{F,c} $$
preserving all root spaces and satisfying 
$\phi(x_\alpha) = \mu_\alpha x_\alpha$ with $\mu_\alpha > 0$
for each $\alpha \in F \subeq \cB$ and $x_{\pm\alpha}$ as above. 
Since $F$ is an integral base of 
$\Delta_F$, we have $\phi\res_{\g_F} = \phi_\chi$ 
for a homomorphism $\chi \: \Spann_\Z \Delta_F \to \C^\times$ 
with $\chi(\alpha) = \mu_\alpha$ for $\alpha \in \cB$ 
(cf.~Lemma~\ref{lem:diag-auto}). Then 
$\im(\chi) \subeq \R^\times_+$, so that $\chi(\alpha) > 0$ for each root. 

For the corresponding hermitian form, we therefore have on 
$\g_\alpha$, $\alpha \in \Delta_F$: 
$$ \kappa_\sigma(x,x) = \kappa(x,\sigma(x)) 
= \chi(\alpha) \kappa(x,\sigma_c(x)) \geq 0 $$
and on $\Spann \check F \subeq \fh$ we have 
$$ \kappa_\sigma(x,x) = \kappa(x,\sigma(x)) 
= \kappa(x,\sigma_c(x)) \geq 0. $$
Therefore $\sigma$ is unitary. 
\end{proof}

\section{Highest weight modules} \label{sec:4}

In this section we construct for each integral weight $\lambda$ of a 
coral locally affine Lie algebra $\g$ a simple module 
$L(\lambda, \fp_\lambda)$ and show that it is integrable 
if $\lambda$ does not vanish on $\z(\g_c)$. 
For complex Lie algebras, we further show that for $\lambda = \lambda^*$, 
the module $L(\lambda, \fp_\lambda)$ is unitary with respect to any 
unitary real form $\fk$ of $\g$.

\begin{definition} Let $(\g,\fh)$ be a split Lie algebra. A triple 
$(\g^+, \g^0, \g^-)$ is said to define a {\it split triangular decomposition}  
if there exist subsets $\Sigma^0, \Sigma^\pm \subeq \Delta$ such that 
\begin{itemize}
\item[\rm(T1)]  $\Delta = \Sigma^+ \dot\cup \Sigma^0 \dot\cup \Sigma^-$ 
is a partition of $\Delta$. 
\item[\rm(T2)]  $\g^\pm = \sum_{\alpha \in \Sigma^\pm} \g_\alpha$ 
and $\g^0 = \fh +\sum_{\alpha \in \Sigma^0} \g_\alpha$. 
\item[\rm(T3)] $[\g^0, \g^\pm] \subeq \g^\pm$. 
\item[\rm(T4)] $\sum_{i = 1}^n \alpha_i \not=0$ for $\alpha_i \in \Sigma^-$ 
and $n > 0$. 
\end{itemize}

Then $\g = \g^- \oplus \g^0 \oplus \g^+$ is a direct sum of vector spaces 
and $\g^\pm \rtimes \g^0$ are subalgebras (sometimes called 
{\it generalized parabolics}). 
\end{definition}

\begin{definition} \label{def:verma} 
Let $(\g^+, \g^0, \g^-)$ be a split triangular decomposition, 
$\fp := \g^0 + \g^+$ and $\lambda \in \h^*$. We extend $\lambda$ 
to a linear functional 
$\lambda \: \fp \to \K$ vanishing on all root spaces. 
We assume that $\lambda([\fp,\fp]) = \lambda([\g^0,\g^0]) =  \{0\}$, so that 
$\lambda$ defines a one-dimensional $\fp_\lambda$-module $\K_\lambda$. 
We write 
$$ M(\lambda,\fp) := \cU(\g)\otimes_{\cU(\fp)} \K_\lambda $$
for the corresponding {\it generalized Verma module} (cf.\ \cite{JK85}). 
This is a $\g$-module generated by a $1$-dimensional $\fp$-module 
$[\1 \otimes \K_\lambda]$ 
isomorphic to $\K_\lambda$, and each other $\g$-module with this 
property is a quotient of $M(\lambda,\fp)$. 

The  Poincar\'e--Birkhoff--Witt Theorem implies that 
the multiplication map \break 
$\cU(\g^-) \otimes \cU(\fp) \to \cU(\g)$
is a linear isomorphism, so that 
$$ M(\lambda,\fp) \cong \cU(\g_-)\otimes_\K \K_\lambda $$
as $\fh$-modules. In particular, $M(\lambda,\fp)$ has an 
$\fh$-weight decomposition, all weights are contained 
in the set 
$\lambda + \Spann_{\N_0} \Sigma^-$,
and (T4) implies that the multiplicity of the weight $\lambda$ is~$1$. 
Therefore 
$M(\lambda, \fp)$ contains a unique maximal submodule, namely 
the sum of all submodules not 
containing $v_\lambda$. This is a submodule whose 
set of weights does not contain $\lambda$, therefore it is proper 
(cf.\ \cite[Prop.~IX.1.12]{Ne00a} and \cite{JK85}). 
We write 
$L(\lambda,\fp)$ for the corresponding unique simple quotient 
and call it the {\it $\fp$-highest weight module} defined by $\lambda$. 
\end{definition}

\subsection*{Highest weight modules for integral weights} 

Let $\g$ be a coral locally extended affine Lie algebra 
and $\lambda \in \fh^*$ an integral weight. We have seen in 
Theorem~\ref{thm:q-grad} that 
$\g^q(\lambda) := \bigoplus_{\lambda(\alpha^\sharp) = q} \g_\alpha$ 
defines a grading of $\g$ 
by a cyclic subgroup of $\Q$. We claim that the three sets 
$$ \Sigma^\pm := \{ \alpha \in \Delta \: 
\pm \lambda(\alpha^\sharp) > 0 \} \quad \mbox{ and } \quad 
\Sigma^0(\lambda) := \{ \alpha \in \Delta \: \lambda(\alpha^\sharp) = 0\}, $$
resp., the corresponding subalgebras 
$$ \g^0(\lambda) := \sum_{\lambda(\alpha^\sharp) = 0} \g_\alpha  
\quad \mbox{ and } \quad 
\g^\pm(\lambda) := \sum_{\pm q >0} \g_q(\lambda) $$
define a split triangular decomposition. Clearly  (T1/2) hold 
by definition, (T3) follows from 
$(\Sigma^0(\lambda) + \Sigma^\pm(\lambda))\cap \Delta 
\subeq \Sigma^\pm(\lambda),$
and (T4) is an immediate consequence of the definition of 
$\Sigma^-(\lambda)$. 

The linear functional $\lambda \in \fh^*$ extends in a natural 
way to a linear functional 
$$ \lambda \: \fp_\lambda := \g^+(\lambda) \rtimes \g^0(\lambda)
=  \sum_{\lambda(\alpha^\sharp) \geq 0} \g_\alpha \to \K $$
vanishing on all root spaces. 
In view of the following lemma, we thus obtain a homomorphism 
of Lie algebras and we obtain from Definition~\ref{def:verma} 
a simple $\g$-module $L(\lambda) := L(\lambda,\fp_\lambda)$. 

\begin{lemma} \label{lem:character} $\lambda \: \fp_\lambda \to \K$ is a homomorphism 
of Lie algebras, i.e., it vanishes on the commutator algebra. 
\end{lemma} 

\begin{proof} In view of $\fh \subeq \fp_\lambda$, the commutator algebra of $\fp_\lambda$ 
is adapted to the 
root decomposition. Therefore it suffices to observe 
that $\lambda$ vanishes 
on 
$$\h \cap [\fp_\lambda,\fp_\lambda] 
= \sum_{\lambda(\alpha^\sharp) = 0} [\g_\alpha, \g_{-\alpha}]
= \sum_{\lambda(\alpha^\sharp) = 0} \K \alpha^\sharp 
$$
(Remark~\ref{rem:1.4}(a)). 
\end{proof}

\begin{definition} \label{def:posyst} 
(a) A $\g$-module $M$ is said to be {\it integrable} 
if it has a weight decomposition with respect to 
$\fh$ and, for each integrable 
root $\alpha \in \Delta_i$, it is a locally finite $\g(\alpha)$-module 
(which is equivalent to $\g_{\pm\alpha}$ acting nilpotently on $M$) 
(cf.\ \cite[Prop.~3.6]{Ka90}).  

(b) A subset $\Delta^+ \subeq \Delta$ is called a {\it positive system} if 
\begin{itemize}
\item[\rm(PS1)] $\Delta = \Delta^+ \cup -\Delta^+$. 
\item[\rm(PS2)] $\sum_{i = 1}^n \alpha_i \not=0$ for $\alpha_i \in \Delta^+$ 
and $n > 0$. 
\end{itemize}

Note that (PS2) implies in particular that 
$\Delta^+ \cap - \Delta^+ = \eset$ and if 
$\alpha,\beta \in \Delta^+$ and  
$\alpha + \beta$ is a root, then it is positive (cf.\ \cite[Lemma~1.2]{Ne98}). 
To each positive system $\Delta^+$ corresponds a split triangular 
decomposition with $\Sigma^\pm = \Delta^\pm$ and $\Sigma^0 = \eset$. 
The corresponding generalized parabolic subalgebra is 
$$ \fp := \fp(\Delta^+) := \fh + \sum_{\alpha \in \Delta^+} \g_\alpha. $$
Since each $\lambda \in \fh^*$ extends to a homomorphism $\fp \to \K$, 
we obtain a simple $\g$-module $L(\lambda, \Delta^+) := L(\lambda, \fp(\Delta^+))$ by the construction in Definition~\ref{def:verma}. 
It is the unique simple $\g$-module of highest weight $\lambda$, i.e., 
generated by a $\fp(\Delta^+)$-weight vector of weight $\lambda$. 
We call these eigenvectors {\it primitive}. 
\end{definition}

\begin{remark} \label{rem:integ} 
Let $M$ be an integrable $\g$-module and  
$\pi \: \g \to \gl(M)$ the corresponding representation. 

(a) The set $\cP_M \subeq \fh^*$ of $\fh$-weights of $M$ 
consists of integral weights because the eigenvalues of 
$\check \alpha$ on $M$ are integral. 

(b) By the definition of the integrable roots, $\g$ is an 
integrable $\g$-module. 

(c) If $\alpha$ is integrable and $x_{\pm \alpha} \in \g_{\pm\alpha}$, 
then the operators $\pi(x_{\pm\alpha})$ are locally nilpotent, so that 
$r_\alpha^M := e^{\pi(x_\alpha)} e^{-\pi(x_{-\alpha})} e^{\pi(x_\alpha)} 
\in \GL(M)$ is defined and satisfies 
$$ \pi(r_\alpha^\g x) = r_\alpha^M\pi(x) (r_\alpha^M)^{-1} 
\quad \mbox{ for } \quad x \in \g$$ 
(cf.\ \cite[Prop.~6.1.3]{MP95}). Since $r_\alpha^\g\res_\fh = r_\alpha$, 
it follows immediately that there exists for each 
$w \in \cW$ an element $w^M \in \GL(M)$ and an automorphism 
$w^\g \in \Aut(\g)$ with $w^\g\res_\fh = w$ and 
$$ \pi(w^\g x) = w^M\pi(x) (w^M)^{-1} \quad \mbox{ for } \quad x \in \g. $$ 
In particular, the representations $\pi$ and its 
$w^\g$-twist $\pi \circ w^\g$ are equivalent. 
This observation also implies that the weight set 
$\cP_M$ is $\cW$-invariant. 
\end{remark}

\begin{prop} \label{prop:weyldep} If 
$\lambda \in \fh^*$ is an integral weight for which 
$L(\lambda, \fp_\lambda)$ is integrable, then 
$$ L(\lambda, \fp_\lambda) \cong L(w\lambda, \fp_{w\lambda}) 
\quad \mbox{ for each } \quad w \in \cW. $$
\end{prop} 

\begin{proof} The main point is that the element $w$ of the Weyl group 
is induced by an automorphism $\phi \in \Aut(\g,\fh,\kappa)$ 
with $(\phi^{-1})^*\mu = w\mu$ for $\mu \in \fh^*$ 
(Remark~\ref{rem:integ}(c)). 
Then 
$$\phi(\fp_\lambda) 
= \fh + \sum_{\lambda(\alpha^\sharp)\geq 0} \g_{(\phi^{-1})^*\alpha} 
= \fh + \sum_{\lambda(\alpha^\sharp)\geq 0} \g_{w\alpha} 
= \fh + \sum_{(w\lambda)(\alpha^\sharp)\geq 0} \g_{\alpha} 
= \fp_{w\lambda}. $$
On $L(\lambda, \fp_{\lambda})$ we now define the $(\phi^{-1})$-twisted 
$\g$-module structure by $x \bullet v := \phi^{-1}(x)v$, for which $[\1\otimes 1]$ 
becomes a $\fp_{w\lambda}$-eigenvector of weight $w\lambda$, 
which leads to an isomorphism to $L(w\lambda, \fp_{w\lambda})$. 
Now the assertion follows from Remark~\ref{rem:integ}(c), 
asserting that the $\phi^{-1}$-twist of 
$L(\lambda, \fp_{\lambda})$ is isomorphic to $L(\lambda, \fp_{\lambda})$. 
\end{proof}

\begin{prop} \label{prop:int-hiwei}
If $\Delta^+$ is a positive system, 
$L(\lambda,\Delta^+)$ is integrable and \break 
$\g^0(\lambda) = \fh + \g(\Delta^\lambda_i)$ holds for 
$\Delta^\lambda_i := \Delta_i \cap \Delta^\lambda$ 
(which is trivially satisfied if $\Delta^\lambda \subeq \Delta_i$), then 
$$ L(\lambda,\Delta^+) \cong L(\lambda, \fp_\lambda). $$
In particular, $L(\lambda, \Delta^+)$ does not depend on the positive 
system $\Delta^+$. 
\end{prop}

\begin{proof}  Let $v_\lambda \in L(\lambda,\Delta^+)$ 
be a generating primitive element. 
Then, for each $\alpha \in \Delta_i$, 
$v_\lambda$ generates an integrable $\g(\alpha)$-module 
of highest weight $\lambda(\check \alpha)$, 
for which $\pm\lambda(\check \alpha)$ are the maximal and minimal 
eigenvalues of $\check \alpha$. If $\lambda(\check\alpha) =0$, 
this implies that $\g(\alpha)v_\lambda=\{0\}$. Now our assumption 
implies that $v_\lambda$ is a $\fp_\lambda$-eigenvector. 
Therefore $L(\lambda, \Delta^+)$ is a simple $\g$-module generated by 
a $\fp_\lambda$-weight vector of weight $\lambda$, 
so that the universal property of $L(\lambda, \fp_\lambda)$ implies 
that it is isomorphic to $L(\lambda,\Delta^+)$. 
\end{proof}

The main feature of the construction in Definition~\ref{def:verma} 
is that it provides a construction of simple ``highest weight'' modules 
without referring to a positive system. 
Proposition~\ref{prop:int-hiwei} 
now tells us that in all classical cases, it produces 
the same result. This provides a new perspective on highest 
weight modules which is more natural for infinite rank algebras, 
because they have no distinguished $\cW$-conjugacy class of 
positive systems (cf.\ \cite{Ne98}). 

\begin{remark} \label{rem:locfin}
  \begin{footnote}
    {Many results stated in this remark have been obtained in 
\cite[Sect.~I]{Ne98} and \cite[Sect.~3]{Ne04} in the context 
of unitary highest weight modules of complex involutive Lie algebras. 
Since we shall need it later, we now explain how one can argue 
in the algebraic context over a general field of characteristic 
zero. }
  \end{footnote}
Let $(\g,\fh)$ be a split Lie algebra and assume that 
all roots are integrable, i.e., $\Delta = \Delta_i$. 
In view of \cite[Thm.~VI.3]{Ne00b}, $\g$ is 
locally finite and its commutator algebra is a direct sum 
of simple split Lie algebras (\cite[Thm.~III.11]{St99}). 

Let $\lambda \in \fh^*$ be an integral weight. Then 
there exists a positive system $\Delta^+ \subeq \Delta$ for which 
$\lambda$ is dominant integral (\cite[Lemma~I.18]{Ne98}), i.e., 
$\lambda(\check \alpha) \in \N_0$ for $\alpha \in \Delta^+$. 

For finite-dimensional reductive split Lie algebras it is well-known 
that the simple highest weight module 
$L(\lambda,\Delta^+)$ is finite-dimensional if and only if 
$\lambda$ is dominant integral and that each 
integrable highest weight module is simple. 
Now $\g = \indlim \g_j$ is a directed union of finite-dimensional 
reductive Lie algebras for which $\fh_j := \fh \cap \g_j$ is a splitting 
Cartan subalgebra of $\g_j$ and whose corresponding root systems are 
$\Delta_j := \{ \alpha \in \Delta \: \g_\alpha \subeq \g_j\}$. 
Then $\Delta_j^+ := \Delta_j \cap \Delta^+$ is a positive system, 
and $\lambda_j := \lambda\res_{\fh_j}$ is dominant integral for $\g_j$. 
For $\g_j \subeq \g_k$, the submodule 
$\cU(\g_j)v_{\lambda_k} \subeq L(\lambda_k,\Delta_k^+)$ 
generated by a primitive element $v_{\lambda_k}$  
is an integrable highest weight module, hence isomorphic 
to $L(\lambda_j,\Delta_j^+)$. We conclude that we may form 
a direct limit module $\indlim L(\lambda_j,\Delta_j^+)$ which 
is a highest weight module of $\g$ of highest weight~$\lambda$. 
As a direct limit of simple $\g_j$-modules, it is simple, hence 
isomorphic to $L(\lambda,\Delta^+)$. This implies in particular 
that 
$$L(\lambda,\Delta^+) \cong \indlim L(\lambda_j, \Delta_j^+)$$ 
is integrable and that its set of weights coincides with 
\begin{equation}
  \label{eq:weightset}
\cP_\lambda := \conv(\cW\lambda) \cap (\lambda + \cQ), 
\end{equation}
where $\cQ = \Spann_\Z \Delta \subeq \fh^*$ is the root group 
(cf.\ \cite[Thm.~I.11]{Ne98}). From the integrability of 
$L(\lambda, \Delta^+)$ and Proposition~\ref{prop:int-hiwei} 
we now derive that 
$$ L(\lambda) = L(\lambda,\fp_\lambda) \cong L(\lambda, \Delta^+)$$ 
does not depend on the choice of $\Delta^+$ 
and Proposition~\ref{prop:weyldep}, together with 
\eqref{eq:weightset}, further shows that 
$L(\lambda)\cong L(\mu)$ if and only if $\mu \in \cW\lambda$
(cf.\ \cite[Thm.~I.20]{Ne98}). 
\end{remark}

The observation summarized in the preceding remark 
was our original motivation to explore the approach to highest weight 
modules of locally affine Lie algebras developed in the present paper. 

The following proposition explains the $\Delta^+$-independent 
picture for highest weight modules of affine Kac--Moody algebras. 

\begin{prop}
  \label{prop:dom-conj}  Let $(\g,\fh,\kappa)$ be an affine Kac--Moody 
algebra, 
$\Pi \subeq \Delta$ a fundamental system of simple roots   
and $\lambda \in \fh^*$. Then the following assertions hold: 
\begin{itemize}
\item[\rm(i)] If $\lambda(\fz(\g))\not=\{0\}$, then 
$\cW\lambda$ contains a unique dominant weight 
$\tilde\lambda$, i.e., 
$\tilde\lambda(\check \alpha) \geq 0$ for $\alpha \in \Pi.$
\item[\rm(ii)] If $\lambda(\check\Delta_i)\not=\{0\}$, then the following 
are equivalent: 
  \begin{itemize}
  \item[\rm(a)] $(\cW \lambda)(\check \alpha)$ is bounded for 
each $\alpha \in \Delta_i$. 
  \item[\rm(b)] $(\cW \lambda)(\check \alpha)$ is bounded for 
some $\alpha \in \Delta_i$. 
  \item[\rm(c)] $\lambda(\fz(\g)) = \{0\}$. 
  \end{itemize}
\item[\rm(iii)] If $\lambda$ is dominant integral, then 
$L(\lambda, \Delta^+) \cong L(\lambda, \fp_{\lambda}).$
\end{itemize}
\end{prop}

\begin{proof} (i) In \cite{Ka90}, a generator of the one-dimensional 
center is denoted $K$. Our assumption implies that $\lambda(K) \not=0$, 
so that the assertion follows from \cite[Prop.~6.6]{Ka90}, combined 
with \cite[Thm.~16]{MP95}.  

(ii) (a) $\Rarrow$ (b) is trivial. 

(b) $\Rarrow$ (c): Pick $m \in \N$ such that 
$\alpha_k := \alpha + k m\delta \in \Delta_i$ holds for each $k \in \N$ 
(cf.\ Lemma~\ref{lem:2.5}) and use \eqref{eq:mul-coroot} to see that 
$$ \check\alpha_k = \check \alpha + \frac{2km}{(\alpha,\alpha)} \delta^\sharp.
$$
Then  
$$ (r_{\alpha_k}\lambda)(\check \alpha) 
= \Big(\lambda -  \big(\lambda(\check \alpha) + \frac{2km}{(\alpha,\alpha)}
\lambda(\delta^\sharp)\big)\alpha_k\Big)(\check \alpha)
= -\lambda(\check \alpha) - \frac{4km}{(\alpha,\alpha)}
\lambda(\delta^\sharp). $$
If the set of these numbers is bounded (for $k \in \N$), then 
$\lambda(\delta^\sharp) =0$, and this is (c) (cf.\ Lemma~\ref{lem:cent}). 

(c) $\Rarrow$ (a): If $\lambda\res_{\fz(\g)} = \{0\}$, 
then $\lambda$ factors through a linear functional 
$$ (\Spann \check\Delta_i)/\fz(\g) 
\cong \Spann \check{\oline\Delta_i}. $$ 
Since the root system $\oline\Delta_i$ is finite and 
$\cW$ acts on it as a finite group, (a) follows. 

(iii) First we recall from \cite[Lemma~10.1]{Ka90} or 
\cite[Prop.~6.1.6]{MP95} that 
$L(\lambda,\Delta^+)$ is integrable if and only if 
$\lambda$ is dominant integral. Next we observe that 
$\Delta^+ \subeq \Spann_{\N_0} \Pi$ 
implies that 
$\Sigma^0(\lambda) = \Delta \cap \Spann_\Z(\Pi_\lambda),$
so that $\g^0(\lambda) = \fh + \g(\Pi_\lambda)$ (for $\Pi_\lambda := \Pi \cap 
\Delta^\lambda$) follows from 
\cite[Prop.~4.1.14]{MP95}. 
Therefore the assumptions of Proposition~\ref{prop:int-hiwei} 
are satisfied, and (ii) follows. 
\end{proof}

\subsection*{Highest weight modules of locally affine Lie algebras} 

We now come to our main results on highest weight modules. 

\begin{theorem} \label{thm:4.10} Let $\g$ be a locally affine Lie algebra 
and $\lambda \in \fh^*$ an integral weight with 
$\lambda(\z(\g_c)) \not=0$, so that $\Delta^\lambda \subeq \Delta_i$. Then 
the following assertions hold: 
\begin{itemize}
\item[\rm(a)] $L(\lambda)$ is an integrable $\g$-module. 
\item[\rm(b)] Its set of weights 
is 
$\cP_\lambda = \conv(\cW\lambda) \cap (\lambda + \cQ)$, 
where $\cQ = \Spann_\Z \Delta.$
\item[\rm(c)] $L(\mu)\cong L(\lambda)$ if and only if 
$\mu \in \cW\lambda$.  
\end{itemize}
\end{theorem}

\begin{proof} (a) We write $\g$ as a 
directed union of subalgebras $\g_F = \fh + \g(\Delta^F_i)$  
as in Proposition~\ref{prop:2.11}, so that $\g_F$ is a direct sum 
of $(\Delta^F)^\bot \subeq \fh$ and the affine Kac--Moody algebra 
$\g(A_{\Pi_F})$. Further, $\fp_\lambda$ is a directed union 
of the subalgebras $\fp_\lambda^F := \fp_\lambda \cap \g_F$. 
In the following, $L(\lambda,\fp_\lambda^F)$ is always understood as a 
$\g_F$-module. We 
now choose an element $w_F \in \cW$ such that 
$\lambda_F := w_F\lambda$ is dominant integral with respect to $\Pi_F$ 
(Proposition~\ref{prop:dom-conj}) 
and observe that, as 
$\g_F$-modules, we have 
$$ L(\lambda, \fp_{\lambda}^F)  
\cong  L(\lambda_F, \fp_{\lambda_F}^F)  
\cong  L(\lambda_F, (\Delta^F)^+) $$
(Propositions~\ref{prop:weyldep} and \ref{prop:dom-conj}(iii)). 
This implies that 
$L(\lambda_F, \fp_{\lambda}^F)$ is an integrable $\g_F$-module 
(\cite[Lemma~10.1]{Ka90} or \cite[Prop.~6.1.6]{MP95}).

From our construction, it follows that for $F_1 \subeq F_2$ 
we have a natural embedding of the simple integrable $\g_{F_A}$-modules 
$$L(\lambda_{F_1}, \fp_{\lambda_{F_1}}^{F_1}) \subeq 
L(\lambda_{F_2}, \fp_{\lambda_{F_2}}^{F_2}) $$
because $U(\g_{F_1})v_{\lambda_{F_2}} \subeq 
L(\lambda_{F_2}, \fp_{\lambda_{F_2}}^{F_2})$ is an integrable 
highest weight module, hence simple (\cite[Cor.~10.4]{Ka90}), 
and therefore isomorphic to 
$L(\lambda_{F_1}, \fp_{\lambda_{F_1}}^{F_1})$. 
We conclude that, for each $\g_F$, 
 $L(\lambda, \fp_\lambda)$ is a direct limit of integrable 
$\g_F$-modules, hence integrable. Since $\g_F$ was arbitrary, 
the assertion follows from the corresponding results in the affine 
case (\cite[Prop.~6.2.7]{MP95}). 

(b) follows immediately from the direct limit description of $L(\mu)$ 
under (a). 

(c) In view of (a), $L(w\lambda) \cong L(\lambda)$ for each 
$w \in \cW$ follows from Proposition~\ref{prop:weyldep}). 
For the converse, we use (b) to see that in the rational 
affine space $\lambda + \Spann_\Q\cQ$, we have 
$$ \Ext(\conv\cP_\lambda) = \Ext(\conv(\cW\lambda)) \subeq \cW\lambda.$$ 
On the other hand, 
$$\cP_\lambda \subeq \lambda + \Spann_{\N_0} \Sigma^-(\lambda) $$
implies that $\lambda \in \Ext(\conv(\cP_\lambda))$, so that the 
$\cW$-invariance of $\cP_\lambda$ implies 
$\cW\lambda = \Ext(\conv(\cP_\lambda))$, and this implies (c). 
\end{proof}

\begin{theorem} Let $\g$ be a locally affine complex Lie algebra 
and $\sigma$ a unitary involution preserving $\fh$ 
(Definition~\ref{def:unitaryrealform}). 
For 
$\mu \in \fh^*$ put $\mu^*(h) := \oline{\mu(\sigma(h))}$.
Let $\lambda = \lambda^*\in \fh^*$ be an integral weight not vanishing 
on the center. Then 
$L(\lambda,\fp_\lambda)$ carries a positive definite hermitian 
form invariant under the unitary real form $\fk$ of $\g$. 
\end{theorem}

\begin{proof} For any affine Kac--Moody Lie algebra $\g_F$, we know from 
\cite[Thm.~11.7]{Ka90} that for each dominant integral weight 
$\lambda_F = \lambda_F^*$, the corresponding integrable weight module 
$L(\lambda_F, \fp_{\lambda_F})$ has a $\fk_F$-invariant positive 
definite hermitian form, which is unique if normalized by 
$(v_{\lambda_F}, v_{\lambda_F}) = 1$ on the highest weight 
vector~$v_{\lambda_F}$. 

Using this uniqueness and the description of $L(\lambda, \fp_\lambda)$ 
as a direct limit of the $\g_F$-modules 
$L(\lambda_F, \fp_{\lambda_F})$ (Theorem~\ref{thm:4.10}), 
it follows that $L(\lambda,\fp_\lambda)$ carries a $\fk$-invariant 
positive definite hermitian form. 
\end{proof}

\section{Appendix 1. Yoshii's classification} \label{sec:5} 

In this appendix we describe Yoshii's classification 
of locally affine root systems of infinite rank and 
show that two reflectable sections are conjugate under the automorphism group. 
We have already seen in Section~\ref{sec:3} that this can be used to 
show that minimal locally affine Lie algebras are determined by 
their root system. Below we describe for each of the seven types 
of root systems a corresponding minimal locally affine Lie algebra 
which is a twisted loop algebra.

We first recall that each irreducible locally 
finite root system of infinite rank is isomorphic to one of the following 
(cf.\ \cite[\S 8]{LN04}). 
Here we realize the root systems in the free vector space 
$\Q^{(J)}$ with basis $\eps_j$, $j \in J$ and the canonical symmetric 
bilinear form defined by $(\eps_i, \eps_j) = \delta_{ij}$: 
\begin{align*}
A_{J} &:= \{ \eps_{j} - \eps_{k} : j,k \in J, j \not= k \}, \\
B_{J} &:= \{ \pm \eps_{j}, \pm \eps_{j} \pm \eps_{k} : j,k \in J, j \not= k \}, \quad (B_{J})_{\rm sh} = \{ \pm \eps_{j}: j \in J\}\\
C_{J} 
&:= \{ \pm 2 \eps_{j}, 
\pm \eps_{j} \pm \eps_{k} : j,k \in J, j \not= k \},\quad 
(C_J)_{\rm lg} = \{ \pm 2 \eps_{j}: j \in J \}\\ 
D_{J} &:= \{ \pm \eps_{j} \pm \eps_{k} : j,k \in J, j \not= k \} 
= (B_J)_{\rm lg} = (C_J)_{\rm sh}.\\ 
BC_{J} &:= \{ \pm \eps_{j}, \pm 2\eps_{j}, 
\pm \eps_{j} \pm \eps_{k} : j,k \in J, j \not= k \}, \quad 
(BC_J)_{\rm ex} = \{ \pm 2\eps_{j}: j\in J\}. 
\end{align*}

For a root system $(V,\cR,(\cdot,\cdot))$, 
we put $\cR^{(1)} := \cR \times \Z  
\subeq V \times \Q$, where the scalar product on $V \times \Q$ is defined 
by $((\alpha,t), (\alpha', t')) := (\alpha,\alpha')$. 
Now we can state Yoshii's classification (\cite[Cor.~13]{YY08}): 

\begin{theorem} \label{thm:yclass} Each irreducible reduced 
locally affine root system $(V,\cR)$ of infinite rank 
is isomorphic to one of the following: 
$A_{J}^{(1)}, B_J^{(1)}, C_J^{(1)}, D_J^{(1)}$, or 
\begin{align*}
B_J^{(2)} &:= \big((B_J)_{\rm sh} \times \Z\big) 
\cup \big((B_J)_{\rm lg} \times 2 \Z\big) 
= \big(B_J\times 2\Z\big) 
\cup \big((B_J)_{\rm sh} \times (2\Z + 1)\big), \\
C_J^{(2)} &:= \big((C_J)_{\rm sh} \times \Z\big) 
\cup \big((C_J)_{\rm lg} \times 2 \Z\big)
= (C_J \times 2\Z) 
\cup \big(D_J  \times (2 \Z+1)\big) \\
(BC)_J^{(2)} &:= 
\Big(\big((BC_J)_{\rm sh} \cup (BC_J)_{\rm lg}\big) \times \Z
\Big)  \cup \big((BC_J)_{\rm ex} \times (2\Z+1)\big)\\
&= (B_J \times 2\Z) \cup \big(BC_J\times (2\Z+1)\big). 
\end{align*}
\end{theorem}

Let $(V,\cR)$ be a locally affine 
root system and recall that a subspace $V' \subeq V$ is called a reflectable 
section 
if $V' \cap \Delta$ maps bijectively onto $\oline\Delta_{\rm red}$. 
From the classification one easily derives the existence of a reflectable  
section. The following theorem proves their uniqueness up to 
conjugacy by automorphisms: 

\begin{theorem}\label{thm:beam-class} If $(V,\cR)$ is a locally affine 
root system of infinite rank 
and $V', V'' \subeq V$ two reflectable sections, then there exists an 
isometric automorphism $\phi\in \Aut(V,\cR)$ 
with $\phi(V') = V''$, inducing the identity on $\oline V$. 
\end{theorem}

\begin{proof} We think of a reflectable section as being realized by a linear 
section $\sigma \: \oline V \to V$ of the quotient map $V \to \oline V$. 
Any other section $\sigma' \: \oline V \to V$ is of the form 
$\sigma' = \sigma + \gamma\cdot \delta$, where $\gamma \: \oline V \to \Q$ 
is a linear functional with 
$\sigma'(\oline\Delta_{\rm red}) \subeq \Delta$. 

We fix a reflectable section $V'$, the corresponding map $\sigma$ and 
the corresponding reduced root system $\Delta_{\rm red} 
= \sigma(\oline\Delta_{\rm red})$. Accordingly, we identify 
$V$ with $\oline V \times \Q$ with $\delta = (0,1)$, so that 
$\Delta \subeq \oline\Delta \times \Z$, as in the classification. 
We now have to determine all other reflectable sections of $\Delta$. 

In all cases, a necessary 
condition on $\gamma$ is $\gamma(\oline\Delta_{\rm red}) \subeq \Z$. 
For the untwisted types $X_J^{(1)}$, this is also sufficient. 
For $B_J^{(2)}$ and $C_J^{(2)}$ we find the conditions 
$\gamma(\alpha) \in \Z$ for $\alpha$ short and 
$\gamma(\alpha) \in 2\Z$ for $\alpha$ long. 
For $BC_J^{(2)}$ we need $\gamma(\alpha) \in \Z$ for $\alpha$ short or long. 

In all these cases, it is easily verified that 
$\phi(\alpha) := \alpha + \gamma(\oline\alpha)\delta$ 
defines an automorphism of $\Delta$ mapping 
$V'$ onto $\sigma'(\oline V)$.  Finally  $\delta \in V^0$ implies 
that $\phi$ is isometric. 
\end{proof}

\begin{remark} We describe for each locally affine root system of  
infinite rank the 
set of all integral weights $\lambda$ with 
$\Delta_{\rm red} \cong 
\Delta^\lambda = \{ \alpha \in \Delta \: \lambda(\alpha^\sharp) = 0\}$, i.e., for which $\Delta^\lambda$ is a reflectable setion.  
We use some information from the proof of Proposition~\ref{prop:3.4}, 
where we have shown that such weights exist. 

We write each root $\beta \in \Delta$ as 
$\beta = n \alpha + m \delta$ with $n =1$ (if $\beta$ is short or long) 
or $n = 2$ (if $\beta$ is extralong and $\alpha$ is short).  
We also normalize the scalar product on roots in such a way that 
long roots have square length $2$. Then short roots have square 
length $1$ (if they occur) and extralong roots have square length $4$. 
In the proof of Proposition~\ref{prop:3.4}
we have seen in \eqref{eq:mul-coroot} that 
$$ (n\alpha + m \delta)\,\check{} 
= \frac{2}{n^2(\alpha,\alpha)}(n\alpha^\sharp  + m \delta^\sharp)
= \frac{1}{n} \check\alpha + \frac{2m}{n^2(\alpha,\alpha)}\delta^\sharp. $$

For the untwisted cases $\cR^{(1)}$, we have 
$\Delta = \Delta_{\rm red} \oplus \Z \delta,$ and  
we find the condition $\lambda(\delta^\sharp) \in \Z$ 
by considering long roots $\alpha$. 

For $B_J^{(2)}$ and $C_J^{(2)}$ 
and a short root $\beta$, $m \in \Z$ is arbitrary, 
which leads to $2\lambda(\delta^\sharp) \in \Z$, and for a 
long root we have $m \in 2 \Z$, which leads to the same condition 
$\lambda(\delta^\sharp) \in \frac{1}{2}\Z$. 

For $BC_J^{(2)}$ we find for extralong roots 
the condition $\lambda(\delta^\sharp) \in 2\Z$, which is also sufficient 
for short and long roots. 
\end{remark}

\subsection*{Realization of minimal locally affine Lie algebras} 

In this subsection we combine Yoshii's classification of locally 
affine root systems (Theorem~\ref{thm:yclass}) with the 
Uniqueness Theorem~\ref{thm:3.4} to realize 
all infinite rank minimal locally affine Lie algebras as 
twisted loop algebras. 

We start with a description of doubly extended loop algebras. 

\begin{example} \label{ex:loop} (cf.\ \cite[\S 5]{MY06}) 
Let $(\oline\g, \oline\fh, \kappa_{\oline\g})$ be 
a split quadratic Lie algebra whose root system $\oline\Delta 
= \oline\Delta_i$ is locally finite and connected, resp., irreducible. 
Further, let $\Gamma \subeq \Q$ be a subgroup containing $1$, and 
$\K[\Gamma]$ be the algebra whose generators we write 
as formal exponentials $t^q$, $q \in \Gamma$. 

(a) We form the Lie algebra 
$\cL^\Gamma(\oline\g) := \K[\Gamma] \otimes \oline\g,$
which is a generalization of a loop algebra (which we obtain for 
$\Gamma = \Z$, for which we simply 
write $\cL(\oline\g)$). It is a $\Gamma$-graded 
Lie algebra with grading spaces $\cL^\Gamma(\oline\g)_q = t^q 
\otimes \oline\g$ and 
$$ \oline\kappa(t^q \otimes x, t^s \otimes y) := 
\delta_{q,-s} \kappa_{\oline\g}(x,y) $$
is a non-degenerate invariant symmetric 
bilinear form on $\cL^\Gamma(\oline\g)$. 
Further, \break 
$D(t^q \otimes x) := q t^q \otimes x$ defines a $\oline\kappa$-skew 
symmetric derivation on 
$\cL^\Gamma(\oline\g)$, so that we may form the associated double extension 
$$ \g := \hat\cL^\Gamma(\oline\g) := 
(\K \oplus_{\omega_D} \cL^\Gamma(\oline\g)) \rtimes_{\tilde D} \K, $$
where $\omega_D(x,y) = \oline\kappa(Dx,y)$ is a $2$-cocycle and 
$\tilde D(z,x) := (0,Dx)$ is the canonical extension of $D$ to 
the central extension $\K \oplus_{\omega_D} \cL^\Gamma(\oline\g)$ 
(cf.\ Definition~\ref{def:doubext}).
Now 
$$ \kappa((z,x,t),(z',x',t')) := zt' + z't + \oline\kappa(x,x') $$
is an invariant symmetric non-degenerate bilinear form on $\g$ and 
$\fh := \K \oplus \oline\fh \oplus \K$
is a splitting Cartan subalgebra, so that 
$(\g,\fh,\kappa)$ is a split quadratic Lie algebra. 
The element $c := (1,0,0)$ is central and 
the eigenvalue of $d := (0,0,1)$ on $t^q \otimes \oline\g$ is $q$. 

It is now easy to verify that 
the root system of $(\g,\fh)$ can be identified with the 
set 
$$ \oline\Delta\times \Gamma \subeq 
\{0\} \times \oline\fh^* \times \K, \quad \mbox{ where } \quad 
(\alpha,q)(z,h,t) := (0,\alpha,q)(z,h,t) = \alpha(h) + t q, $$
and that the set of integrable roots is 
$\Delta_i = \oline\Delta \times \Gamma.$

For root vectors 
$x_{(\alpha,q)} = t^q \otimes x_\alpha \in \g_{(\alpha, q)}$ 
with $[x_\alpha,x_{-\alpha}] = \check \alpha$, we have  
$$ [t^q \otimes x_\alpha, t^{-q} \otimes x_{-\alpha}] 
= (q \kappa_{\oline\g}(x_\alpha, x_{-\alpha}), \check \alpha) 
= \Big(\frac{2q}{(\alpha,\alpha)}, \check \alpha\Big) $$
(cf.\ Remark~\ref{rem:1.4}). Since $(\alpha,q)$ takes the value $2$ 
on this element, it follows that 
$$ (\alpha,q)\check{} 
= \Big(\frac{2q}{(\alpha,\alpha)}, \check \alpha\Big)
\quad \mbox{ and } \quad 
\kappa((\alpha,q)\check{}, (\beta,r)\check{}) 
= \kappa_{\oline\g}(\check \alpha, \check \beta). $$
From that we easily derive for the scalar product of the roots 
$$ ((\alpha, q), (\beta, r)) = (\alpha,\beta), $$
which implies that $(\g,\fh,\kappa)$ is a 
a LEALA. 

(b) Writing a linear function on $\fh$ as a triple 
$\lambda = (z,\lambda_0, t) \in \K \times \oline\fh^* \times \K$, 
we conclude that $\lambda$ is integral 
if and only if 
$$ \frac{2qz}{(\alpha,\alpha)} + \lambda_0(\check \alpha) \in \Z$$ 
holds for each $q \in \Gamma$ and $\alpha \in \oline\Delta$. 
This means that $\lambda_0 \in \oline\fh^*$ 
is an integral weight of $\oline\g$ and 
$z \in \frac{(\alpha,\alpha)}{2q} \Z$
for each $q \in \Gamma$. The latter condition has a non-zero solution $z$ 
if and only if the subgroup $\Gamma$ is cyclic. We conclude that 
there are integral weights $\lambda$ not vanishing on the 
central element $(1,0,0)$ if and only if $\Gamma \cong \Z$, which 
corresponds to the classical case of loop algebras 
(cf.\ Theorem~\ref{thm:2.8}). 
\end{example}

\begin{remark}
If $\Gamma$ is not cyclic, then the 
group $\Gamma$ is a directed union of cyclic infinite groups 
$\Z q_j$, so that 
the Lie algebra $\g = \cL^\Gamma(\oline\g)$ 
is a direct limit of doubly extended loop algebras 
isomorphic to $\hat\cL(\oline\g)$. 
If $\oline\g$ is finite-dimensional, this exhibits $\hat\cL^\Gamma(\oline\g)$ 
as a direct limit of affine Kac--Moody algebras, but it is not locally 
affine in the sense of Definition~\ref{def:locafflie} (cf.\ \cite{YY08}). 
\end{remark}

If $X_J \in \{ A_J, B_J,C_J, D_J\}$ is one of the irreducible 
locally affine Lie algebras, then we have the following corresponding 
locally affine simple Lie algebras. 

For a set $J$ and a field $\K$, 
we write $\gl_J(\K)$ for the set of all 
$(J \times J)$-matrices with finitely many non-zero entries, i.e., 
the finitely supported functions on $J \times J$. 
Then the set $\fh = \Spann_\K \{ E_{jj} \: j \in J \}$ of diagonal 
matrices is a splitting Cartan subalgebra with the root system 
$A_J$, where $\eps_j(E_{kk}) := \delta_{jk}$. Its commutator algebra  
is the simple Lie algebra $\fsl_J(\K)$. 

Next, let $2J := J \dot{\cup} (- J)$ be a disjoint union, 
where $-J$ denotes a copy of the set $J$
whose elements are denoted by $-j$, $j \in J$. 
We define $S \in \K^{2J \times 2J}$ (the set of all $J \times J$-matrices)  
by 
$$ S_\pm := \pmat{ \0 & \1 \\ \pm\1 & \0} 
= \sum_{j \in J} (E_{j, -j} \pm E_{-j, j}).$$
Then 
$$ \sp_{2J}(\K) := \{ x \in \gl_{2J}(\K) \: x^\top S_- + S_- x = 0\} $$
and 
$$ \fo_{2J}(\K) := \{ x \in \gl_{2J}(\K) \: x^\top S_+ + S_+ x = 0\}$$ 
are split Lie algebras with respect to the space 
$$ \fh = \Spann \{ E_{jj} - E_{-j,-j} \: j \in J\} $$
of diagonal matrices. If we define 
$\eps_j(E_{kk} - E_{-k,-k}) := \delta_{jk}$, then the 
corresponding root systems are 
$C_J$ for $\sp_{2J}(\K)$ and $D_J$ for $\fo_{2J}(\K)$. 

To realize the root system $B_J$, 
we put $2J+1 := 2J \dot{\cup} \{0\}$ (disjoint union) 
and 
$$ S := \pmat{ \0 & 0 & \1 \\ 0 & 1 & 0 \\ \1 & 0 & \0} 
= E_{00} + \sum_{j \in J} (E_{j, -j} + E_{-j, j}).$$
Then 
$$ \fo_{2J+1}(\K) := \{ x \in \gl_{2J+1}(\K) \: x^\top S + S x = 0\}$$ 
is a split Lie algebra with respect to 
$\fh = \Spann \{ E_{jj} - E_{-j,-j} \: j \in J\}$ and the 
root system $B_J$. Since the quadratic spaces 
$(\K^{2J}, S_+)$ and $(\K^{2J+1}, S)$ are isomorphic, 
the Lie algebras $\fo_{2J+1}(\K)$ and $\fo_{2J}(\K)$ 
are isomorphic, although they have two non-isomorphic root 
decompositions with respect to non-conjugate Cartan subalgebras 
(cf.\ \cite[Lemma~I.4]{NS01}). 
 
On all these Lie algebras, there is a natural non-degenerate 
invariant symmetric bilinear form, given by 
$\kappa(x,y) := \tr(xy)$.

To obtain realizations of minimal locally affine Lie algebras, 
we now turn to twisted loop algebras. 
Let $(\oline\g,\oline\fh, \kappa_{\oline\g})$ 
be one of the four types of 
simple locally finite split quadratic 
Lie algebras with root system of type $X_J$. 
Further, let $\sigma \in \Aut(\oline\g)$ be an involutive automorphism 
fixing $\oline\fh$. Then $\sigma$ induces an automorphism of the 
root system which is isometric because of the positive definiteness 
of the form and the fact that every homomorphism 
$\Z/2\Z \to \R^\times_+$ is trivial. This 
implies that $\oline\kappa$ is $\sigma$-invariant. 

Let $\oline\g = \oline\g_+ \oplus \oline\g_-$ 
be the $\sigma$-eigenspace decomposition of 
$\oline\g$, and put $\oline\fh_\pm := \oline\fh \cap \oline\g_\pm$. 
{\sl We assume that $\fh_+$ is maximal abelian in $\oline\g_+$}, hence a 
splitting Cartan subalgebra and write 
$\oline\Delta_\pm \subeq \oline\h_+^*$ for the set 
of non-zero weights 
of $\oline\g_+$, resp., the set of $\oline\fh_+$-weights in 
$\oline\g_-$. 

Define $\tilde\sigma \in \Aut(\cL(\oline\g))$ by 
$\tilde\sigma(t^q \otimes x) := (-1)^q t^q \otimes \sigma(x)$ 
and consider the corresponding twisted loop algebra 
$$ \cL(\oline\g, \sigma) := \{ \xi \in \cL(\oline\g) \: 
\tilde \sigma(\xi) = \xi\} 
= \big(\K[t^{\pm 2}] \otimes \oline\g_+\big) 
\oplus \big(t\K[t^{\pm 2}] \otimes \oline\g_-\big). $$
This Lie algebra is invariant under the canonical derivation 
$D$ of the loop algebra, so that we also obtain a corresponding double 
extension 
$\g := \hat\cL(\oline\g, \sigma) \subeq \hat\cL(\oline\g),$
which is the set of fixed points for the involution on  
$\hat\cL(\oline\g, \sigma)$, defined by 
$\hat\sigma(z,\xi,t) := (z, \tilde\sigma(\xi),t)$ 
(which makes sense because $\tilde\sigma$ leaves $\oline\kappa$ 
invariant). 

The subalgebra $\fh := \K \oplus \oline\fh_+ \oplus \K$ 
is a splitting Cartan subalgebra of $\g$ and the restriction of 
the quadratic invariant form of $\hat\cL(\oline\g)$ is 
non-degenerate on $\g$. Its root system is given by 
$$ \Delta_i 
= (\oline\Delta_+ \times 2\Z) \dot\cup (\oline\Delta_- \times (2\Z+1)). $$

In the proof of the classification theorem, we need the following 
elementary geometric lemma. 

\begin{lemma}\label{lem:quad} Let $(V,\beta)$ be a quadratic space,  
$v \in V$ be non-isotropic and 
$$ g(x) := x - \frac{2\beta(v,x)}{\beta(v,v)} v $$
be the orthogonal reflection in the hyperplane $v^\bot$. 
Then $\Ad(g)X := gXg^{-1}$ is an involutive 
automorphism of $\fo(V,\beta)$, and for the 
corresponding eigenspaces $\fo(V,\beta)_{\pm 1}$, we have 
$$ \fo(V,\beta)_1 \cong \fo(v^\bot, \beta),  $$
and the map 
$$ \phi \: v^\bot \to \fo(V,\beta)_{-1}, \quad 
\phi(x) := \beta_{v,x} - \beta_{x,v}, \quad \beta_{v,w}(u) := \beta(v,u)w,$$ 
is a linear isomorphism. 
\end{lemma}

\begin{proof} Let $V_\pm := V_\pm(g)$ denote the eigenspaces of $g$, 
so that $V_+ = v^\bot$ and $V_{-} = \K v$. Then 
$\fo(V,\beta)_1$ consists of all skew-symmetric linear maps 
commuting with $g$, i.e., preserving both $g$-eigenspaces. 
In view of $\fo_1(\K) = \{0\}$, this leads to the given description 
of $\fo(V,\beta)_1$. 

On the other hand, the relation $\Ad(g)X = -X$ is equivalent to 
$X V_\pm \subeq V_\mp$. In view of $\beta_{v,w}^\top = \beta_{w,v}$, 
the image of $\phi$ lies in $\fo(V,\beta)$ and it clearly maps 
$v^\bot$ into $\K v$ and $v$ into $v^\bot$. Conversely, 
let $\gamma \in \fo(V,\beta)_{-1}$. 
Then $\gamma(v) \in v^\bot$ and we put 
$x := \beta(v,v)^{-1}\gamma(v)$. We claim that 
$\gamma = \phi(x)$. Clearly, 
$$\gamma(v) = \beta(v,v)x = \beta_{v,x}(v) = \phi(x)(v). $$
For $y \in v^\bot$ we have $\gamma(y) \in \K v$ and 
$$ \beta(v,\gamma(y)) 
= - \beta(\gamma(v),y) 
= - \beta(v,v) \beta(x,y) 
= - \beta(v, \beta_{x,v}(y))
= \beta(v, \phi(x)(y)), $$
which implies that $\gamma(y) = \phi(x)(y)$, and hence that 
$\gamma = \phi(x)$. 
\end{proof}

\begin{theorem} \label{thm:classi} For the irreducible reduced 
locally affine root systems of infinite rank, the corresponding 
minimal locally affine Lie algebras can be constructed as follows:  
\begin{itemize}
\item[\rm(i)] For the root systems of type $X_J^{(1)}$ and a 
simple split Lie algebra $(\oline\g,\oline\fh)$ with root system 
$X_J$, the doubly extended loop algebra 
$\hat\cL(\oline\g)$ is minimal locally affine with the 
root system $X_J^{(1)} = X_J \times \Z$ . 
\item[\rm(ii)] For the root systems of type $X_J^{(2)}$,  
the doubly extended twisted loop algebra 
$\hat\cL(\oline\g, \sigma)$ is minimal locally affine with the 
root system $X_J^{(2)}$, where 
\begin{itemize}
\item[\rm $B_J^{(2)}$:] $\oline\g := \fo_{2J'}(\K)$ with 
$J' := J \cup \{j_0\}$, $j_0 \not\in J$ and 
$\sigma = \Ad(g)$ for the orthogonal reflection in the hyperplane  
$(e_{j_0} - e_{-j_0})^\bot \subeq  \K^{2J'}$.\begin{footnote}{Note that our description of the Lie algebra 
of type $B_J^{(2)}$ is more explicit than the one in \cite{YY08}.} 
\end{footnote}
\item[\rm $C_J^{(2)}$:] $\oline\g = \fsl_{2J}(\K)$ and 
$\sigma(x) = -S x^\top S^{-1}$, where 
$S = \pmat{\0 & \1 \\ -\1 & \0} \break 
= \sum_{j \in J} (E_{j, -j} - E_{-j, j})$. 
\item[\rm $BC_J^{(2)}$:] $\oline\g = \fsl_{2J+1}(\K)$ and 
$\sigma(x) = - S x^\top S^{-1}$, where \hfill\break  
$S := \pmat{ \0 & 0 & \1 \\ 0 & 1 & 0 \\ \1 & 0 & \0} 
= E_{j_0,j_0} + \sum_{j \in J} (E_{j, -j} +  E_{-j, j}).$
\end{itemize}
\end{itemize}
\end{theorem}

\begin{proof} (i) follows immediately from Example~\ref{ex:loop}. 

(ii) $B_J^{(2)}$: 
In $\oline\g$ we consider the canonical Cartan subalgebra  
$$\oline\fh = \Spann \{ E_{jj} - E_{-j,-j} \: j \in J'\}.$$ 
Then 
$v := e_{j_0} - e_{-j_0} \in \K^{2J'}$ is a non-isotropic vector 
defining an orthogonal reflection $g \in \OO_{2J'}(\K)$ in $v^\bot$, 
and we obtain an involution of $\fo_{2J'}(\K)$ by 
$\sigma(x) := \Ad(g)x = gxg^{-1}$. 
It is easy to verify that $\Ad(g)$ preserves $\oline\fh$ with 
$$ \oline\fh_+ = \Spann \{ E_{jj} - E_{-j,-j} \: j \in J\} 
\quad \mbox{ and } \quad 
\oline\fh_- = \K (E_{j_0,j_0} - E_{-j_0,-j_0}).$$ 
From Lemma~\ref{lem:quad} we now derive that 
$$\oline\g_+ \cong \fo_{2J+1}(\K) \quad \mbox{ and } \quad 
\oline\Delta_+ = B_J. $$
This lemma also shows that 
$\oline\g_- \cong v^\bot \cong \K^{2J} \oplus \K(e_{j_0} - e_{-j_0})$, 
so that the set of non-zero weights of $\oline\fh_+$ is 
$\oline\Delta_- = \{ \pm \eps_j \: j \in J\} = (B_J)_{\rm sh}$, and this 
leads to 
$$ \Delta = (B_J \times 2 \Z) \cup \big((B_J)_{\rm sh} \times 
(2\Z+ 1)\big). $$

$C_J^{(2)}$: We have $\fsl_{2J}(\K)_+ = \sp_{2J}(\K)$ with the 
Cartan subalgebra  
$$\fh_+ = \Spann\{ E_{jj} - E_{-j,-j} \: j \in J\}.$$ 
The condition $\sigma(x) = - x$ is equivalent to $(Sx)^\top = - Sx$, 
which for $x = \pmat{a & b \\ c & d}$ is equivalent to 
$a^\top = d, b^\top = - b$ and $c^\top = -c$. From that it is easy to 
see that $\oline\Delta_-$ is the root system $D_J$, so that 
$$\Delta 
= (C_J \times 2\Z) \cup (D_J \times (2\Z + 1)) = C_J^{(2)}. $$
The corresponding minimal locally affine Lie algebra is the doubly 
extended twisted loop algebra $\hat\cL(\fsl_{2J}(\K), \sigma)$. 

$BC_J^{(2)}$: In this case 
$\fsl_{2J+1}(\K)_+ = \fo_{2J+1}(\K)$
and 
$$\fh_+ = \Spann \{ E_{jj} - E_{-j,-j} \: j \in J\}$$ is a splitting 
Cartan subalgebra of $\fo_{2J+1}(\K)$ for which the root system 
is $B_J$. 

The condition $\sigma(x) = - x$ is equivalent to $(Sx)^\top = Sx$. 
Evaluating this condition by writing 
$x$ as a $(3 \times 3)$-block matrix according to the 
decomposition $2J + 1 = J \cup \{0\} \cup -J$, 
we see that $\oline\Delta_-$ is the root system $BC_J$, so that 
$$\Delta 
= (B_J \times 2\Z) \cup (BC_J \times (2\Z + 1)) = BC_J^{(2)}. $$
The corresponding minimal locally affine Lie algebras is the doubly 
extended twisted loop algebra $\hat\cL(\fsl_{2J+1}(\K), \sigma)$. 
\end{proof}

\begin{remark} \label{rem:cj-alternate} 
For $C_J^{(2)}$ we also describe an alternative 
realization, which is a geometric variant of Kac' approach via diagram 
automorphisms which is more inplicit (cf.\ \cite{Ka90}). 

On $\oline\g := \fsl_{2J}(\K)$ we consider the involutive automorphism 
defined by 
$\sigma(x) = -S x^\top S^{-1}$, where 
$$S = \pmat{\0 & \1 \\ \1 & \0} 
= \sum_{j \in J} (E_{j, -j} + E_{-j, j}).$$ 
Then $\fsl_{2J}(\K)_+ = \fo_{2J}(\K)$ with the Cartan subalgebra 
$$\fh_+ = \Spann\{ E_{jj} - E_{-j,-j} \: j \in J\}.$$ 
The condition $\sigma(x) = - x$ is equivalent to $(Sx)^\top = Sx$, 
which for $x = \pmat{a & b \\ c & d}$ is equivalent to 
$a^\top = d, b^\top = b$ and $c^\top = c$. From that it is easy to 
see that $\oline\Delta_-$ is the root system $C_J$, so that 
$$\Delta 
= (D_J \times 2\Z) \cup (C_J \times (2\Z + 1)) 
= (D_J \times \Z) \cup ((C_J)_{\rm lg} \times (2\Z + 1)). $$
Then $\oline\Delta = C_J$, but $D_J \times \{0\}$ does not correspond 
to a reflectable section. To obtain a reflectable section, we consider instead the 
hyperplane 
$$ V' := \Spann \{ (2\eps_j, 1) \: j \in J\}, $$
which leads to $\Delta_{\rm red} = C_J$ and 
$$\Delta 
\cong (C_J \times 2\Z) \cup (D_J \times (2\Z + 1)) 
= C_J^{(2)}. $$
The corresponding minimal locally affine Lie algebra is the doubly 
extended twisted loop algebra $\hat\cL(\fsl_{2J}(\K), \sigma)$. 
\end{remark}

\section{Appendix 2. Isomorphisms of twisted loop algebras} 

Let $\oline\g$ be a $\K$-Lie algebra and 
$\sigma \in \Aut(\oline\g)$ an automorphism with 
$\sigma^m = \id_{\oline\g}$. Suppose that $\K$ contains a 
primitive $m$-th root of unity $\zeta \in \K^\times$, i.e., 
$\ord(\zeta) = m$. 
We define $\tilde\sigma \in \Aut(\cL(\oline\g))$ by 
$\tilde\sigma(t^q \otimes x) := \zeta^q t^q \otimes \sigma(x)$ 
and consider the corresponding twisted loop algebra 
$$ \cL(\oline\g, \sigma) := \{ \xi \in \cL(\oline\g) \: 
\tilde \sigma(\xi) = \xi\}. $$

\begin{lemma} \label{lem:maxideal} 
Let $(\oline\g,\oline\fh)$ be a locally finite split simple 
Lie algebra and $\cL(\oline\g) = \K[t^\pm] \otimes \oline\g$ be the 
corresponding loop algebra. 
Then the following assertions hold: 
\begin{itemize} 
\item[\rm(i)] Each ideal of $\cL(\oline\g)$ is of the form 
$I \otimes \oline\g$ for an ideal $I \trile  R$. 
\item[\rm(ii)] If $\K$ is algebraically closed, then 
the maximal ideals of $\cL(\oline\g)$ are the kernels of the 
evaluation maps $\ev_z \: \cL(\oline\g) \to \oline\g$, $z \in \K^\times$, 
sending $r \otimes x$ to $r(z)x$. 
\end{itemize}
\end{lemma}

\begin{proof} (i) First we note that $\oline\g$ is a central 
simple $\oline\g$-module, 
so that $\cL(\oline\g)$ is an isotypic semisimple $\oline\g$-module of 
type $\oline\g$. Let $R := \K[t^\pm]$ be the ring of Laurent polynomials. 
Then $\cL(\oline\g) = R \otimes_\K \oline\g$ and we may identity 
$R$ with the multiplicity space $\Hom_{\oline\g}(\oline\g, \cL(\oline\g))$ 
by assigning to $r \in R$ the embedding 
$x \mapsto r \otimes x$. 
In fact, let $\psi \in \Hom_{\oline\g}(\oline\g, \cL(\oline\g))$ 
and $0 \not=x \in \oline\g$. We write 
$\psi(x) = \sum_i r_i \otimes y_i$ with linearly independent elements 
$r_i \in R$ and $y_i \in \oline\g$ 
and observe that this implies that 
$$ \psi(\oline\g) = \psi(\cU(\oline\g)x) 
\subeq \cU(\oline\g)(\sum_i r_i \otimes \oline\g) 
\subeq \sum_i r_i \otimes \oline\g,$$ 
where we use the canonical action of the enveloping algebra 
$\cU(\oline\g)$ on $\cL(\oline\g)$. 
We derive the existence 
of $\phi_i = \lambda_i \id \in \End_{\oline\g}(\oline\g) \cong \K$  
with 
$\phi(z) = \sum_i r_i \otimes \phi_i(z)$ for each $z \in \oline\g$, 
and this leads to 
$\psi(z) = (\sum_i \lambda_i r_i) \otimes z$ for each $z \in \oline\g$.

We conclude that each simple $\oline\g$-submodule of $\cL(\oline\g)$ 
is of the form $r \otimes \oline\g$, and since each submodule 
is semisimple, hence a sum of simple submodules, it is of the form 
$M \otimes \oline\g$ for a unique subspace $M \subeq R$. 

Assume, in addition, that $M \otimes \oline\g$ is an ideal of $\cL(\oline\g)$. 
Then 
$$ M \otimes \oline\g \supeq [t \otimes \oline\g, M \otimes \oline\g] 
= tM \otimes [\oline\g,\oline\g] 
= tM \otimes \oline\g  $$
implies that $tM \subeq M$, and we likewise obtain $t^{-1}M = M$, 
showing that $M \trile R$ is an ideal. 

(ii) If $\K$ is algebraically closed, then the maximal ideals 
of $\K[t^\pm]$ are the kernels of the point evaluations 
$\ev_z$, $z \in \K^\times$, so that the assertion follows from (i). 
\end{proof}

\begin{prop} \label{prop:recon} Let $(\oline\g_j,\oline\fh_j)$, $j=1,2$, 
be locally finite split simple 
Lie algebras, $m \in \N$, and $\sigma_j \in \Aut(\oline\g_j)$ 
be automorphisms with $\sigma_j^m = \id_{\oline\g_j}$. 
Then 
$$ \cL(\oline\g_1,\sigma_1) \cong
\cL(\oline\g_2,\sigma_2) \quad \Rarrow \quad 
\oline\g_1 \cong \oline\g_2. $$
\end{prop}

\begin{proof} Let $\oline\K$ denote the algebraic closure of $\K$. 
If we can prove the assertion for the Lie algebras 
$\oline\K \otimes_\K \oline\g_j$, 
then we arrive at an isomorphism 
$\oline\K \otimes_\K \oline\g_1 \cong \oline\K \otimes_\K \oline\g_2$, 
so that the classification of locally finite split simple 
Lie algebras implies that $\oline\g_1 \cong \oline\g_2$ because 
the isomorphism class is determined by the type of the corresponding 
root systems (cf.\ \cite[Thm.~VI.7]{NS01}). 
We may therefore assume that $\K$ is algebraically closed. 

Let $S := \K[t^\pm]$ and let $R := \K[t^{\pm m}]$ be the 
subring generated by $t^{\pm m}$. 
According to \cite[Lemma~4.3]{ABP04}, $\cL(\oline\g_j, \sigma_j)$ 
is central over $R$ and an $S/R$-form of $R \otimes_\K \oline\g$, i.e., 
$$ S \otimes_R \cL(\oline\g_j, \sigma_j) \cong S \otimes \oline\g_j. $$

Since $\K$ is algebraically closed, each element of $\K$ has 
$m$-th roots, so that \cite[Thm.~4.6]{ABP04} shows that 
\begin{equation}
  \label{eq:R-iso}
\cL(\oline\g_1,\sigma_1) \cong_\K \cL(\oline\g_2,\sigma_2) 
\quad \Rarrow \quad 
\cL(\oline\g_1,\sigma_1) \cong_R \cL(\oline\g_2,\sigma_2). 
\end{equation}
This in turn leads to 
$$ S \otimes_\K \oline\g_1 
\cong_S S \otimes_R \cL(\oline\g_1,\sigma_1) 
\cong_S  S \otimes_R \cL(\oline\g_2,\sigma_2) 
\cong_S  S \otimes_\K \oline\g_2. $$

Finally Lemma~\ref{lem:maxideal} shows that all quotients of 
$S \otimes_\K \oline\g_j$ by maximal ideals are isomorphic to 
$\oline\g_j$, so that we obtain $\oline\g_1 \cong_\K \oline\g_2$.
\end{proof}

\begin{theorem} \label{thm:1stclassi} We have isomorphism between the 
minimal locally affine Lie algebras corresponding 
to the following pairs of root systems: 
$$ (B_J^{(1)}, D_J^{(1)}), \quad 
  (C_J^{(2)}, BC_J^{(2)}) \quad \mbox{ and } \quad 
  (B_J^{(1)}, B_J^{(2)}). $$
\end{theorem}

\begin{proof} (a) From the isomorphism of the Lie algebras 
$\fo_{2J}(\K) \cong \fo_{2J+1}(\K)$ (\cite[Lemma~I.4]{NS01})
and the fact that any isomorphism 
is (up to a factor) isometric with respect to the invariant 
quadratic form, it follows that the corresponding doubly extended 
loop algebras are also isomorphic. Therefore the non-isomorphic 
root systems $B_J^{(1)}$ and $D_J^{(1)}$ correspond to isomorphic 
minimal locally affine Lie algebras. 

(b) In Remark~\ref{rem:cj-alternate}, we have also seen how to 
realize the root system $C_J^{(2)}$ by a twisted loop algebra 
$\hat\cL(\fsl_{2J}(\K), \sigma)$, where 
$\fsl_{2J}(\K)_+  \cong \fo_{2J}(\K)$. From 
the isomorphism $(\K^{(2J)},\beta_1) \cong (\K^{(2J+1)},\beta_2)$ 
of quadratic spaces, with 
$$ \beta_j(x,y) = x^\top S_j y, \quad S_1 = \pmat{\0 & \1 \\ \1 & \0} 
\quad \mbox{ and } 
\quad  S_2 = \pmat{ \0 & 0 & \1 \\ 0 & 1 & 0 \\ \1 & 0 & \0}, $$ 
we obtain an isomorphism 
$(\fsl_{2J}(\K), \sigma_1) \cong 
(\fsl_{2J+1}(\K), \sigma_2)$ of Lie algebras with 
involution.  
Combining Theorem~\ref{thm:classi} with 
Remark~\ref{rem:cj-alternate}, it now follows that 
the minimal locally affine Lie algebra 
$\hat\cL(\fsl_{2J+1}(\K), \sigma_2)$ 
of type $BC_J^{(2)}$ 
is isomorphic to the minimal locally affine Lie algebra 
$\hat\cL(\fsl_{2J}(\K), \sigma_1)$ of type $C_J^{(2)}$. 

(c) We realize $B_J^{(2)}$ as in Theorem~\ref{thm:classi} via the 
quadratic space $(V = \K^{(2J')},\beta')$, where 
$$ \beta'(x,y) = \sum_{j \in J'} x_j y_{-j} + x_{-j}y_j $$
and $\sigma = \Ad(g)$, where $g$ is the orthogonal reflection in 
$v := e_{j_0} - e_{-j_0}$. 

Next we choose an orthogonal decomposition 
$V = \K v \oplus V_1 \oplus V_2$, 
where $V_k = (\K^{2J_k}, \beta_k)$, $k =1,2$, 
and 
$\beta_k(x,y) = \sum_{j \in J_k} x_j y_{-j} + x_{-j}y_j$
is the canonical form on $\K^{2J_k}$. 
Accordingly, we obtain a decomposition 
$$ V = \K v \oplus 
(\K^{(J_1)} \oplus \K^{(-J_1)}) \oplus
(\K^{(J_2)} \oplus \K^{(-J_2)}),$$ 
so that we may represent linear maps on $V$ accordingly by 
$(5 \times 5)$-block matrices. 
We thus obtain a group homomorphism 
$$ \alpha \: \K^\times \to \OO(V,\beta), \quad 
\alpha(t) := \pmat{
1 & 0 & 0 & 0 & 0 \\ 
0 & t\1 & \0 & \0 & \0 \\ 
0 & \0 & t^{-1}\1 & \0 & \0 \\ 
0 & \0 & \0 & \1 & \0 \\ 
0 & \0 & \0 & \0 & \1}. $$

Thinking of elements $\xi$ of $\cL(\fo_{2J'}(\K),\sigma)$ 
as maps $\K^\times \to \fo_{2J'}(\K)$ satisfying 
$$ \xi(-t) = \Ad(g)(\xi(t))\quad \mbox{ for } \quad t \in \K^\times, $$
it is now easy to see that 
$$ \xi \mapsto \xi', \quad \xi'(t) := \Ad(\alpha(t))(\xi(t)) $$
defines an isomorphism of Lie algebras 
$$ \cL(\fo_{2J'}(\K),\sigma) \to \cL(\fo_{2J'}(\K),\sigma'), $$
where 
$\sigma' = \Ad(\alpha(-1))\Ad(g) = \Ad(\alpha(-1)g) = \Ad(g')$, 
where 
$$ g' := \alpha(1)g = \pmat{
-1 & 0 & 0 & 0 & 0 \\ 
0 & -\1 & \0 & \0 & \0 \\ 
0 & \0 & -\1 & \0 & \0 \\ 
0 & \0 & \0 & \1 & \0 \\ 
0 & \0 & \0 & \0 & \1} $$
is the orthogonal reflection in the subspace 
$V_2 \subeq V$. Next we observe that the triple
$(V,\beta,g')$ of a quadratic space with an orthogonal reflection 
is isomorphic to the triple 
$(V_1 \oplus V_2, \beta_1 \oplus \beta_2, \alpha(-1))$, so that 
$$ \cL(\fo_{2J'}(\K),\sigma') 
\cong \cL(\fo_{2(J_1+J_2)}(\K),\alpha(1)). $$
Reversing the argument above, we further derive 
$$ \cL(\fo_{2(J_1+J_2)}(\K),\alpha(1)) 
\cong \cL(\fo_{2(J_1+J_2)}(\K),\id)
\cong \cL(\fo_{2J}(\K)). $$
\end{proof}

For the proof of the following theorem, we recall some 
facts on automorphisms of $\fsl_J(\K)$ from \cite{St01}: 

\begin{remark} \label{rem:nina-aut} 
A matrix $A \in \K^{J \times J}$ defines a linear 
endomorphism of the free vector space $\K^{(J)}$ if and only if each 
column has only finitely many non-zero entries. We write 
$\GL_J(\K)_f \subeq \GL(\K^{(J)})$ for the subgroup of those linear 
automorphisms 
$\phi(x) = Ax$, $A \in \K^{J \times J}$, for which the adjoint 
map, which is represented by the transposed matrix $A^\top$, preserves 
the subspace $\K^{(J)}$ of $(\K^{(J)})^* \cong \K^J$. 

It is shown in \cite{St01} that an automorphism of $\fsl_J(\K)$ 
either is of the form $\phi_A(x) = AxA^{-1}$ for some 
$A \in \GL_J(\K)_f$, or of the form 
$\tilde\phi_A(x) := -Ax^\top A^{-1}$ for some $A \in \GL_J(\K)_f$. 

Both types of automorphisms can easily be distinguished 
by their action on the invariant polynomial 
$p_3(x) := \tr(x^3)$ of degree~$3$, which is non-zero for $|J| > 2$. 
In fact, $p_3$ is invariant under automorphisms of the form 
$\phi_A$ and $p_3 \circ \tilde\phi_A = - p_3$. 
\end{remark}

\begin{theorem} \label{thm:2nd-classi} The minimal locally affine 
Lie algebras corresponding to the root systems 
$A_J^{(1)}$ and $C_J^{(2)}$ are not isomorphic. 
\end{theorem}

\begin{proof} In view of Theorem~\ref{thm:classi}, 
it suffices to show that the Lie algebras 
$\cL(\fsl_{J}(\K))$ and 
$\cL(\fsl_{J}(\K),\sigma)$ are not isomorphic if 
$\sigma$ is an involutive automorphism of the form 
$\sigma(x) = \tilde\phi_S(x) = - S x^\top S^{-1}$, 
where $S$ is any matrix 
defining an involutive automorphism of $\fsl_J(\K)$.  
After base field extension to the algebraic closure $\oline\K$ of $\K$, we may 
w.l.o.g.\ assume that $\K$ is algebraically closed. 

We argue by contradiction. If 
$\cL(\fsl_{J}(\K)) \cong \cL(\fsl_{J}(\K),\id) 
\cong \cL(\fsl_{J}(\K),\sigma)$, then 
\cite[Thm.~IV.6]{ABP04} implies that these Lie algebras are isomorphic 
over the ring $R := \K[t^{\pm 2}] \subeq S := \K[t^{\pm}]$. 
We also recall from \cite[Lemma~IV.3]{ABP04} that 
$$ S \otimes_R \cL(\fsl_J(\K), \sigma) 
\cong \cL(\fsl_J(\K), \sigma) \oplus t \cdot \cL(\fsl_J(\K), \sigma)
= S \otimes_\K \fsl_J(\K) = \cL(\fsl_J(\K)). $$
We therefore obtain an $S$-automorphism 
$\phi \in \Aut(\cL(\fsl_J(\K)))$, mapping the $R$-subalgebra 
$\cL(\fsl_J(\K),\sigma)$ to 
$\cL(\fsl_J(\K),\id) = R \otimes_\K \fsl_J(\K)$. 

From Lemma~\ref{lem:maxideal}, we know that 
the maximal ideals of $\cL(\fsl_J(\K))$ all have the form 
$S_z \otimes \fsl_J(\K) = S_z\cL(\fsl_J(\K))$, with 
$S_z := \{ f \in S \: f(z) = 0\}$ for some $z \in \K^\times$. 
Since $\phi$ is $S$-linear, it therefore preserves all maximal 
ideals of $\cL(\fsl_J(\K))$, hence induces for each $z \in \K^\times$ 
an automorphism $\phi_z\in \Aut(\oline\g)$ via 
$\phi_z(x) = \phi(1 \otimes x)(z)$, which in turn implies 
$$ \phi(f)(z) = \phi_z(f(z)) \quad \mbox{ for } z \in \K^\times, 
f \in \cL(\fsl_J(\K)). $$
Let $\tilde\sigma(f)(z) := \sigma(f(-z))$, so that 
$$\cL(\fsl_J(\K),\sigma) = \cL(\fsl_J(\K))_+ 
\quad \mbox{ and } \quad 
t\cL(\fsl_J(\K),\sigma) = \cL(\fsl_J(\K))_- $$
is the eigenspace decomposition of $\tilde\sigma$. 
Likewise 
$$\cL(\fsl_J(\K)) = \big(R \otimes \fsl_J(\K)\big) 
\oplus \big(t R \otimes \fsl_J(\K)\big) $$
is the eigenspace decomposition of the involution defined by 
$\tilde\id(f)(z) := f(-z)$. 
Since $\phi$ maps the $\tilde\sigma$-eigenspaces to the 
corresponding $\tilde\id$-eigenspaces, we obtain 
$$ \phi \circ \tilde\sigma = \tilde\id \circ \phi, $$
which leads to 
$$ \phi_z \circ \sigma = \phi_{-z} \quad \mbox{ for } \quad z \in  \K^\times,$$
and hence to the factorization 
$$ \sigma = \phi_{-1} \circ \phi_1^{-1}. $$

Pick $x \in \fsl_J(\K)$ with $p_3(x) := \tr(x^3)\not=0$. 
Then the function $\K^\times \to \K$, \break $z \mapsto p_3(\phi_z(x))$ 
is a Laurent polynomial, and we know from 
Remark~\ref{rem:nina-aut} that its only possible values are 
$\pm p_3(x)$, so that it is constant. This in turn implies that 
all automorphisms $\phi_z$ either fix $p_3$ or reverse its sign. 
In particular, $p_3$ is invariant under $\sigma$, but this contradicts 
$\sigma = \tilde\phi_S$ (cf.~Remark~\ref{rem:nina-aut}). 
\end{proof}

Combining the results from the preceding two theorems 
with Proposition~\ref{prop:recon} and the classification 
for the locally finite case in \cite{NS01}, we finally obtain 
the following classification of minimal locally affine Lie algebras: 

\begin{theorem}[Classification Theorem] \label{thm:3rd-classi} 
For each infinite set $J$, there are four isomorphism 
classes of minimal locally affine Lie algebras with 
$|\Delta| = |J|$. They are represented by the split Lie algebras 
with the root systems 
$A_J^{(1)}, B_J^{(1)}, C_J^{(1)}$ and $C_J^{(2)}$, resp., 
the loop algebras $\cL(\oline\g)$ with $\g$ of type 
$A_J$, $B_J$ or $C_J$, and 
the twisted loop algebra $\cL(\fsl_{2J}(\K),\sigma)$ with 
$\sigma(x) = -S x^\top S^{-1}$ and 
$S = \pmat{\0 & \1 \\ -\1 & \0}$. 
\end{theorem}

\begin{remark} In \cite{Sa08}, Salmasian deals with the 
closely related conjugacy problem for 
maximal abelian splitting subalgebras of loop algebras 
of the form $\fL(\fk)$, where $\fk$ is coral locally finite. 
It turns out that the Lie algebras corresponding to the 
root systems $A_J$ and $C_J$ yield only one conjugacy class, but for 
type $B_J$ and $D_J$, it is only shown that the number of conjugacy classes 
is $\leq 5$. Since $B_J$ and $D_J$ correspond to isomorphic Lie algebras, 
it is $\geq 2$. 
\end{remark}


\bibliographystyle{amsalpha}

\begin{thebibliography}{aaaaaaaa} 

\bibitem[AA-P97]{AA-P97} Allison, B., Azam, S., Berman, S., Gao, Y., and A.\
Pianzola, ``Extended affine Lie algebras and their root systems,'',
Memoirs of the Amer.\ Math.\ Soc. {\bf 603}, 1997

\bibitem[ABGP97]{ABGP97} Allison, B. N., 
S.\ Berman, Y. Gao, and A.~Pianzola, {\it A characterization 
of affine Kac--Moody Lie algebras}, Comm. Math. Phys. {\bf 185:3} (1997), 
671--688 

\bibitem[ABP04]{ABP04} Allison, B. N., 
S.\ Berman, and A.~Pianzola, {\it Covering Algebras II: 
Isomorphism of loop algebras}, J. reine angew. Math. {\bf 571} (2004), 
39--71 



\bibitem[GK81]{GK81} Gabber, O., 
and V.\ G.\ Kac., {\it On defining relations of certain
infinite dimensional Lie algebras}, Bull of the Amer.\ Math.\
Soc. {\bf 5} (1981), 185--189 

\bibitem[HoG04]{HoG04} Hofmann, G., ``The Geometry of Reflection Groups,'' 
Thesis, TU Darmstadt, Shaker Verlag, 2004 


\bibitem[JK85]{JK85} Jakobsen, H.~P., and V.\ Kac, 
{\it A new class of unitarizable
highest weight representations of infinite-di\-men\-sio\-nal Lie algebras}, 
in ``Non--linear equations in classical and quantum field theory,''
N.\ Sanchez ed., Springer Verlag, Berlin,
Heidelberg, New York, Lecture Notes in Physics {\bf 226} (1985), 1--20


\bibitem[Ka90]{Ka90} Kac, V., ``Infinite-dimensional Lie Algebras," Cambridge University Press, $3^{rd}$ printing, 1990  

\bibitem[KN01]{KN01} K\"urner, B., and K.-H. Neeb, {\it 
Invariant symmetric bilinear forms for reflection groups}, J. geom. 
{\bf 71} (2001), 99--127 

\bibitem[LN04]{LN04} Loos, O., and E. Neher, ``Locally finite root systems,'' 
Memoirs of the Amer. Math. Soc., Vol. 171, {\bf 811}, 2004 


\bibitem[Mac72]{Mac72} Macdonald, I. G., {\it Affine root systems and 
Dedekind's $\eta$-Function}, Invent. Math. {\bf 15} (1972), 91--143 

\bibitem[MR85] {MR85} A. Medina and P. Revoy, {\it Alg\`ebres de Lie et produit scalaire 
invariant}, Ann. scient. \'Ec. Norm. Sup. $4^e$ s\'erie {\bf 18}(1985), 
533-561

\bibitem[MP95]{MP95} Moody, R., and A. Pianzola, ``Lie algebras with triangular
decompositions", Canad. Math. Soc. Series of Monographs and advanced texts, 
Wiley In\-ter\-science, 1995

\bibitem[MY06]{MY06} Morita, Y., and Y.\ Yoshii, {\it 
Locally extended affine Lie algebras}, J. Algebra {\bf 301} (2006), 59--81 

\bibitem[MY08]{MY08} ---, {\it Locally loop algebras and locally 
affine Lie algebras}, in preparation 

\bibitem[Ne98]{Ne98} Neeb, K.-H., {\it Holomorphic highest weight representations
of infinite dimensional complex classical groups}, 
J.\ Reine Angew. Math.\ {\bf 497} (1998), 171--222  

\bibitem[Ne00a]{Ne00a} ---, ``Holomorphy and Convexity in Lie Theory,'' 
Expositions in Mathematics {\bf 28}, de Gruyter Verlag, Berlin, 2000  

\bibitem[Ne00b]{Ne00b} ---, {\it Integrable roots in split graded 
Lie algebras}, J. Algebra {\bf 225} (2000), 534--580 

\bibitem[Ne01]{Ne01} ---, {\it Borel--Weil theory for loop groups}, in 
``Infinite Dimensional K\"ahler Manifolds'', 
Eds. A. Huckleberry, T. Wurzbacher, DMV-Seminar {\bf 31}, 
Birkh\"auser Verlag, 2001; 179--229  

\bibitem[Ne04]{Ne04} ---, {\it Infinite-dimensional Lie groups and their 
representations}, in ``Lie Theory: Lie Algebras and Representations,'' 
Progress in Math. {\bf 228}, Ed. J.~P.~Anker, B. \O{}rsted, 
Birkh\"auser Verlag, 2004; 213--328

\bibitem[NS01]{NS01} Neeb, K.--H., and N.\ Stumme, {\it The classification of
locally finite split simple Lie algebras}, J. Algebra {\bf 553} (2001), 
25--53 

\bibitem[Neh08]{Neh08} Neher, E., {\it Extended affine Lie algebras and 
other generalizations of affine Lie algebras---a survey}, 
in ``Developments and trends in infinite dimensional Lie theory'', 
Eds. K.-H. Neeb and A. Pianzola, Progress in Math., Birkh\"auser 
Verlag, to appear 

\bibitem[PS86]{PS86} Pressley, A., and G. Segal, ``Loop Groups," Oxford University Press, 
Oxford, 1986

\bibitem[Sa08]{Sa08} Salmasian, H., {\it Conjugacy of maximal toral 
subalgebras of direct limits of loop algebras}, Preprint 2008 

\bibitem[St99]{St99} Stumme, N., {\it The structure of locally finite split Lie
  algebras}, Journal of Algebra {\bf 220} (1999), 664--693 

\bibitem[St01]{St01} ---, {\it Automorphisms and conjugacy 
of compact real forms of the classical infinite dimensional matrix 
Lie algebras}, Forum Math. {\bf 13:6}  (2001),  817--851

\bibitem[YY08]{YY08} Y. Yoshii, {\it Locally extended affine root 
systems}, in this volume: ``Quantum affine algebras, extended affine Lie 
algebras and applications'', Eds. Y. Gao et. al., Contemp. Math., 
to appear 

\end{thebibliography}

\end{document}